\documentclass[11pt]{article}

\pdfoutput=1

\usepackage{mathrsfs}
\usepackage{amsmath}
\usepackage{amsfonts}
\usepackage{graphicx}
\usepackage{color}
\usepackage{bbm}
\usepackage{amsthm}
\usepackage{amssymb}

\newtheorem{theorem}{Theorem}[section]
\newtheorem{definition}[theorem]{Definition}
\newtheorem{lemma}[theorem]{Lemma}
\newtheorem{remark}[theorem]{Remark}

\begin{document}
\title{Problem of eigenvalues of stochastic Hamiltonian systems with boundary conditions and Markov chain}
\author{Tian Chen\footnote{Shandong University, Jinan, China,  chentian43@sdu.edu.cn} \hspace{1cm} Xijun Hu\footnote{Shandong University, Jinan, China,  xjhu@sdu.edu.cn} \hspace{1cm} Zhen Wu\footnote{Shandong University, Jinan, China, wuzhen@sdu.edu.cn }}

\maketitle

\vspace{2mm}\noindent\textbf{Abstract.}\quad In this paper, we study the eigenvalue problem of stochastic Hamiltonian system driven by Brownian motion and Markov chain with boundary conditions and time-dependent coefficients. For any dimensional case, the existence of the first eigenvalue is proven and the corresponding eigenfunctions are constructed by virtue of dual transformation and generalized Riccati equation system. Furthermore, we have more finely characterized the existence of all eigenvalues and constructed the related eigenfunctions for one-dimensional Hamiltonian system. Moreover, the increasing order of these eigenvalues have also been given.

\vspace{2mm} \noindent\textbf{Keywords.}\  Stochastic Hamiltonian system, Markov chain, eigenvalue problem, Riccati equation.

\vspace{2mm}\noindent\textbf{MSC2020.}\  60J10, 34B99, 34L15


\section{Introduction}
As we all know, the linear backward stochastic differential equation (BSDE) was firstly introduced by Bismut \cite{B1978} as the adjoint equation to solve the optimal control problem. It was not until 1990 that Pardoux and Peng \cite{PP1990} gave the general form and the well-posedness of BSDE. Independently, Duffie and Epstein \cite{DE1992} also introduce the BSDE in the economic context, they present a stochastic differential representation of recursive utility, which is generalization of the standard additive utility. Due to the widespread application of BSDE in the field of economics (see \cite{EPQ1997}), it has attracted widespread attention from scholars \cite{BDHPS2003,DMP1996,S1997}.

Stochastic Hamiltonian system, a fully coupled forward-backward stochastic differential equation (FBSDE), was originally introduced in the stochastic optimal control problem as a necessary condition, the stochastic maximum principle (SMP), for optimality, see \cite{B1973,B1981}. FBSDE is formed by coupling a SDE and a BSDE, which appears in various fields, for example, stochastic optimal control \cite{P1990,YZ1999,W2013,CWW2022}, and partial differential equation (PDE) \cite{P1991,PT1999}. The well-posedness of FBSDE is the key to the research in the above-mentioned issues, Antonelli first studied such a problem in \cite{A1993} by virtue of a fixed point theorem on a relatively small time interval. For the case where the time interval is arbitrarily fixed, there are three methods to study this problem. The first one is introduce by Ma et al. \cite{MPY1994} and called the four-step scheme approach. The well-posedness can be obtained by the relationship between the Markovian FBSDE and PDE. The second one is the method of continuation under the Monotonicity condition, which was established by Hu and Peng  \cite{HP1995} and developed by \cite{Y1997,PW1999}. The last one is given by Pardoux and Tang \cite{PT1999}. By virtue of the discounting method and fixed point theorem under the weakly coupled condition. Later, Ma et al. \cite{MWZZ2015} proposed a unified framework approach to solve the one-dimensional FBSDE in a general non-Markovian structure.

All literature works above concern the well-posedness of FBSDE as well as stochastic Hamiltonian system. Peng \cite{P2000} first studied the problem of eigenvalues of stochastic Hamiltonian systems with boundary conditions. The goal of this problem is to find some real number $\rho_i\in \mathbb R$, $i=1,2,\cdots$, such that the parameterized stochastic Hamiltonian system admits a trivial solution, i.e. $(x_t,y_t,z_t)=(0,0,0)$, and some non-trivial solution. Here, $\rho_i$ is called eigenvalues and the corresponding non-trivial solutions are called eigenfunctions of this Hamiltonian system. Inspired by these results, Wang and Wu \cite{WW2017} investigated this problem of the Hamiltonian systems driven by Poission process and Brownian motion. Recently, Jing and Wang \cite{JW2023} extend the related results in \cite{P2000} from time-invariant case to time-dependent case. Moreover, they present the order of growth for these eigenvalues.

Recently, due to the widespread application of Markov chains in related fields such as economics and engineering, it has attracted the attention of scholars. Compared with the traditional model, regime-switching system can better describe the transition of system state \cite{Zhang2001,ZY2003}. At the same time, the SMP of regime-switching system has also been further developed. Donnelly \cite{D2011} proved a sufficient SMP of regime-switching diffusion model, Tao and Wu \cite{TW2012} gave the necessary and sufficient SMP of forward-backward regime-switching system, interested readers also may refer \cite{ZES2012,SGZ2018}.

In this paper, we investigate a class of eigenvalues problem of the stochastic Hamiltonian system with Markov chain and time-dependent coefficients matrix. The eigenvalue problem of time-variant system is more complex than that of time-invariant system, thus we need to characterize it more finely. In order to investigate this problem, we introduce the dual Legendre transformation of general Hamiltonian system with Markov chain and present the relationship between the Hamiltonian system and a generalized Ricccati equations system. Based on these results, the existence of the first eigenvalue can be proved and the related eigenfunctions can be constructed for multi-dimensional case. Furthermore, for one-dimensional case, we present all eigenvalue located in $(\rho_b,+\infty)$ and construct all corresponding eigenfunctions. Moreover, we also study the existence of all eigenvalue in $\mathbb R$ under some sharper assumption and present the order of growth for these eigenvalues.

The remaining sections are organized as follows. Section 2 gives some preliminary and formulates the eigenvalue problem of the stochastic Hamiltonian system driven by Brownian motion and Markov chain. We introduced the dual transformation of Hamiltonian system with Markov chain in general and linear cases in Section 3.  In Section 4, for arbitrarily dimensional case, we give the existence of the first eigenvalue and constructed the related eigenfunctions. For one-dimensional case, all eigenvalues located in $(\rho_b,+\infty)$ are given and the related eigenfunctions are constructed in Section 5. And we present the increasing order of these eigenvalues.

\section{Preliminary and Problem Formulation}
Given $T>0$. Let $(\Omega,\mathcal F, \mathbb{F},\mathbb{P})$ be the filtered complete probability space with an one-dimensional standard Brownian motion $W$ and a continuous-time Markov chain $\alpha$ valued in $\mathcal{D}=\{1,2,\dots,m\}$, and suppose they are mutual independent. Here, $\mathcal{F}_t$ is generated by $\{W_s,\alpha_s\}_{0\leq s\leq t}$ augmented by all $\mathbb P$-null set $\mathcal{N}$, and define $\mathbb F^{W}\triangleq \{\mathcal F^W_t\}_{0\leq t\leq T}$ with $\mathcal F^W_t\triangleq \sigma\{W_s\}_{ 0\leq s\leq t} \vee \mathcal N$, $\mathbb F^{\alpha}\triangleq \{\mathcal F^{\alpha}_t\}_{0\leq t\leq T}$ with $\mathcal F^{\alpha}_t=\sigma\{\alpha_s\}_{0\leq s\leq t}\vee \mathcal N$. $\mathbb{R}^m$ is the $m$-dimensional Euclidean space, $|\cdot|$ and $\langle \cdot,\cdot\rangle$ denote the norm and scalar product in the Euclidean space, $D^{\top}$ ($D^{-1}$) denotes the transposition (reverse) of matrix $D$. $\mathbb S^n$ ($\mathbb{S}^n_+$) denotes the set of symmetric (positive semi-definite) $n\times n$ matrices. Denote $M>(\geq,\gg) 0$, when $M\in \mathbb{S}^n$ is positive definite (positive semi-definite, uniformly positive definite).

In addition, we assume that the generator of $\alpha$ is $Q=(q_{ij})_{m\times m}$. Note that $q_{ij}\geq 0$, for $i\ne j$ and $\sum_{j=1}^mq_{ij}=0$, so $q_{ii}\leq 0$. In what follows, we further assume that $q_{ij}>0$ for $i\ne j$, so $q_{ii}<0$. Define function $f^i : D \rightarrow R, f^i_x=I_{\{i\}}(x)$, $i\in D$, which has the following semimartingale decomposition,
\begin{eqnarray*}
  \begin{aligned}\label{fdecom}
    f^i_{\alpha_t}&=f^i_{\alpha_0} + \int_0^t\sum_{j=1}^m q_{\alpha_s j}f^i_j\mathrm{d}s+M^i_t
    =f^i_{\alpha_0}+\int_0^t q_{\alpha_si}\mathrm{d}s+M_t^i,
  \end{aligned}
\end{eqnarray*}
where $M^i$ is an $\mathbb F$ martingale satisfying $ E|M^i_t|^2\! <\! \infty$. Define $ V_t^{ij}\!=\!\sum_{0<s\leq t}f^i_{\alpha_{s-}}f^j_{\alpha_s}$, $i\neq j$, which counts the number of times that $\alpha$ jumps from $i$ to $j$ up to time $t$. Since $i\neq j$, we have
\begin{eqnarray*}
  \begin{aligned}
    V^{ij}_t &= \sum_{0<s\leq t}f^i_{\alpha_{s-}}f^j_{\alpha_s}=\sum_{0<s\leq t}f^i_{\alpha_{s-}}\left(f^j_{\alpha_s} -f^j_{\alpha_{s-}}\right)
    =\sum_{0<s\leq t}f^i_{\alpha_{s-}}\left(\Delta f^j_{\alpha}\right)_s\\
    &=\int_0^tf^i_{\alpha_{s-}}\mathrm{d}f^j_{\alpha_s}
    =\int_0^t f^i_{\alpha_s}q_{\alpha_sj}\mathrm{d}s+ \int_0^tf^i_{\alpha_{s-}}\mathrm{d}M^j_s.
  \end{aligned}
\end{eqnarray*}
We give the notation $\sum_{i\neq j}^n$ as abbreviation for $\sum_{j=1}^n\sum_{i=1,i\neq j}^n$. Then we define
\begin{eqnarray*}
  \begin{aligned}
    V_t:= \sum_{i\neq j}^n V^{ij}_t =\int_0^t \sum_{i\neq j}^n \Big(q_{\alpha_sj}f^i_{\alpha_s}\mathrm{d}s+f^i_{\alpha_{s-}}\mathrm{d}M_s^j\Big)
    =\int_0^t r_s\mathrm{d}s+\tilde{V}_t,
  \end{aligned}
\end{eqnarray*}
where $r_s=\sum_{i\neq j}^m q_{\alpha_sj}f^i_{\alpha_s}$, $\tilde{V}_t=\sum_{i\neq j}^m \int_0^tf^i_{\alpha_{s-}}\mathrm{d}M_s^j$. 


Now, we introduce some notations which will be used in this
paper. For any given Hilbert space $\mathbb{H}$ and filtration $\mathbb F$,
let $L^2_{\mathcal F_T}(\Omega;\mathbb{H})$ be the space of $\mathbb{H}$-valued $\mathcal{F}_T$-measurable square-integrable random variables; $L^{\infty}(0,T;\mathbb H)$ be the space of $\mathbb{H}$-valued deterministic uniformly bounded functions; $L^2_{\mathbb F}(0,T;\mathbb H)$ be the space of $\mathbb{H}$-valued $\mathcal{F}_t$-adapted square-integrable processes; $F^2_{\mathbb F}(0,T;\mathbb H)$ be the space of $\mathbb{H}$-valued $\mathcal{F}_t$-predictable square-integrable processes; $S^2_{\mathbb F}(0,T;\mathbb H)$ be the space of $\mathbb{H}$-valued $\mathcal{F}_t$-adapted c\`adl\`ag processes satisfying $\mathbb{E}[\sup_{0\leq t\leq T}|x_t| ^2]<\infty$; $M^2_{\mathbb F}(0,T;\mathbb H)$ be the space of $\mathbb{H}$-valued $\mathcal{F}_t$-predictable processes satisfying $\mathbb{E}[\int_0^T|x_t|^2r_t\mathrm{d}t ]<\infty$; $C(0,T;\mathbb S^n)$ be the space of $\mathbb S^{n}$-valued continuously differential functions.

\subsection{Preliminary}

In this subsection, we first give the existence and uniqueness result of the FBSDE with regime-switching. Consider the following FBSDE,
\begin{equation}\label{RSFBSDE}
  \left\{
  \begin{aligned}
    &\mathrm{d}x_t=b(t,x_t,y_t,z_t,\theta_{t})\mathrm{d}t+\sigma(t,x_t,y_t,z_t,\theta_{t})\mathrm{d}B_t+\gamma(t,x_{t-},y_{t-},z_t,\theta_{t})\mathrm{d}\tilde V_t,\\
    &\mathrm{d}y_t=-f(t,x_t,y_t,z_t,\theta_{t})\mathrm{d}t+z_t\mathrm{d}B_t+\theta_{t}\mathrm{d}\tilde V_t,\\
    &x_0=x_0,\qquad y_T=\Phi(x_T),
  \end{aligned}
  \right.
\end{equation}
where $b,\sigma,\gamma,f:[0,T]\times\Omega\times \mathbb R^n\times \mathbb R^n\times \mathbb R^n\times \mathbb R^{n}\rightarrow \mathbb R^n$ are all progressive processes and $\Phi:\Omega\times \mathbb R^n\rightarrow \mathbb R^n$ is $\mathcal F_T$ measurable process. Let
\begin{equation*}
  \begin{aligned}
    F=\left(\begin{matrix}
      -f\\
      b\\
      \sigma\\
      \gamma
    \end{matrix}\right):(t,x,y,z,\theta)\in [0,T]\times \Omega\times \mathbb R^n\times \mathbb R^n\times \mathbb R^n\times \mathbb R^{n}\rightarrow \mathbb R^{4n}.
  \end{aligned}
\end{equation*}
And then we introduce the following assumption.

\noindent\textbf{Assumption (H1)} The function $F$ and $\Phi$ satisfy

\noindent (i) For any $\xi:=(x,y,z,\theta)\in \mathbb R^n\times \mathbb R^n\times \mathbb R^n\times \mathbb R^{n}$, $F(\cdot,\xi)\in L^2_{\mathbb F}(0,T,\mathbb R^{4n})$;

\noindent (ii) There exists a constant $C>0$ such that
\begin{equation*}
  \begin{aligned}
    &|F(t,\xi_1)-F(t,\xi_2)|\leq C|\xi_1-\xi_2|,\quad \forall \xi_1,\xi_2\in \mathbb R^{4n},\\
    &|\Phi(x_1)-\Phi(x_2)|\leq C|x_1-x_2|,\quad \forall x_1,x_2\in \mathbb R^n.
  \end{aligned}
\end{equation*}

\noindent (iii) There exists a constant $\beta>0$ such that
\begin{equation*}
  \begin{aligned}
    &\langle F(t,\xi_1)-F(t,\xi_2),\xi_1-\xi_2\rangle \leq -\beta|\xi_1-\xi_2|^2,\quad \forall \xi_1,\xi_2\in \mathbb R^{4n},\\
    &\langle \Phi(x_1)-\Phi(x_2),x_1-x_2\rangle \geq 0,\quad \forall x_1,x_2\in \mathbb R^n.
  \end{aligned}
\end{equation*}

Inspired by \cite{HP1995,PW1999}, we have the following theorem.
\begin{theorem}\label{thm21}
  Under Assumption (H1), regime-switching FBSDE \eqref{RSFBSDE} admits a unique solution $(x,y,z,\theta)\in L^2_{\mathbb F}(0,T;\mathbb R^n)\times L^2_{\mathbb F}(0,T;\mathbb R^n)\times F^2_{\mathbb F}(0,T;\mathbb R^n)\times M^2_{\mathbb F}(0,T;\mathbb R^{n})$.
\end{theorem}

\subsection{Problem Formulation}
In this subsection, we can study the boundary problem of stochastic Hamiltonian system with regime-switching by virtue of Theorem \ref{thm21}. Let $h(x,y,z,\theta): \mathbb R^n\times \mathbb R^n\times \mathbb R^{n}\times \mathbb R^{n}\rightarrow \mathbb R$ be a $\mathcal C^1$ real function of $(x,y,z,\theta)$, which is called a Hamiltonian function, and let $\Phi(i,x):\mathcal D\times \mathbb R^n\rightarrow \mathbb R$ be a $\mathcal C^1$ real function of $x$. The problem is to find a quadruple $(x,y,z,\theta)\in L^2_{\mathbb F}(0,T;\mathbb R^n)\times L^2_{\mathbb F}(0,T;\mathbb R^n)\times F^2_{\mathbb F}(0,T;\mathbb R^n)\times M^2_{\mathbb F}(0,T;\mathbb R^{n})$ satisfying the following stochastic Hamiltonian system,
\begin{equation}\label{Hamiltonian}
  \left\{
  \begin{aligned}
    &\mathrm{d}x_t=\partial_y h(t,x_t,y_t,z_t,\theta_{t})\mathrm{d}t+\partial_z h(t,x_t,y_t,z_t,\theta_{t})\mathrm{d}B_t+\partial_{\theta} h({t},x_{t-},y_{t-},z_t,\theta_{t})\mathrm{d}\tilde V_t,\\
    &\mathrm{d}y_t=-\partial_x h(t,x_t,y_t,z_t,\theta_{t})\mathrm{d}t+z_t\mathrm{d}B_t+\theta_{t} \mathrm{d}\tilde V_t,\\
    &x_0=x_0,\qquad y_T=\partial_x\Psi(x_T),
  \end{aligned}
  \right.
\end{equation}
where $\partial_x h$, $\partial_y h$, $\partial_z h$ and $\partial_{\theta} h$ are respectively gradients of $h$ with respect to $x$, $y$, $z$ and $\theta$. Obviously, FBSDE \eqref{Hamiltonian} is a special case of \eqref{RSFBSDE} with
\begin{equation*}
  \begin{aligned}
    F=\left(\begin{matrix}
      -\partial_x h\\
      \partial_y h\\
      \partial_z h\\
      \partial_{\theta} h
    \end{matrix}\right)\qquad \text{and} \qquad \Phi=\partial_x\Psi.
  \end{aligned}
\end{equation*}
If $F$ and $\Phi$ satisfy Assumption (H1), then we can know that the Hamiltonian system \eqref{Hamiltonian} admits a unique solution by Theorem \ref{thm21}.

Now, we are ready to formulate the eigenvalue problem for stochastic Hamiltonian system with Markov chain. Suppose that $\bar h(t,x,y,z,\theta):[0,T]\times  \mathbb R^n\times \mathbb R^n\times \mathbb R^n\times \mathbb R^{n} \rightarrow \mathbb R$ is a $\mathcal C^1$ function of $(x,y,z,\theta)$, and denote for each $\rho\in \mathbb R$,
\begin{equation}
  h^{\rho}(t,x,y,z,\theta)=h(t,x,y,z,\theta)+\rho\bar h(t,x,y,z,\theta).
\end{equation}
In addition, we assume that for $(x,y,z,\theta)=(0,0,0,0)$,
\begin{equation*}
  \partial_xh=\partial_yh=\partial_zh=\partial_{\theta}h=\partial_x\bar h=\partial_y\bar h=\partial_z\bar h=\partial_{\theta}\bar h=0.
\end{equation*}
Consider the following parameterized stochastic Hamiltonian system with Markov chain,
\begin{equation}\label{FBSDE0}
  \left\{
  \begin{aligned}
    &\mathrm{d}x_t=\partial_y h^{\rho}(t,x_t,y_t,z_t,\theta_t)\mathrm{d}t+\partial_z h^{\rho}(t,x_t,y_t,z_t,\theta_{t})\mathrm{d}B_t
    +\partial_{\theta} h^{\rho}(t,x_{t-},y_{t-},z_t,\theta_{t})\mathrm{d}\tilde V_t,\\
    &\mathrm{d}y_t=-\partial_x h^{\rho}(\alpha_t,x_t,y_t,z_t,\theta_{1,t},\theta_{2,t})\mathrm{d}t+z_t\mathrm{d}B_t+\theta_{1,t} \mathrm{d}\tilde V_t^1+\theta_{2,t} \mathrm{d}\tilde V_t^2,\\
    &x_0=0,\qquad y_T=0.
  \end{aligned}
  \right.
\end{equation}
It is easy to see that $(x,y,z,\theta)=(0,0,0,0)$ is a trivial solution of FBSDE \eqref{FBSDE0}. The eigenvalue problem is to find some real number $\rho\in \mathbb R$ such that \eqref{FBSDE0} admits nontrivial solutions. In this paper, we focus on the case that $h$ and $\bar h$ are in the form of
\begin{equation*}
  h(t,\xi)=\frac{1}{2}\langle H(t)\xi,\xi\rangle,\qquad \bar h(t,\xi)=\frac{1}{2}\langle \bar H(t)\xi,\xi\rangle,\qquad \forall \xi:=(x,y,z,\theta)\in \mathbb R^{4n}.
\end{equation*}
where $H(\cdot), \bar H(\cdot)\in C(0,T;\mathbb S^{4n})$ has the following form,
\begin{equation}\label{coefficient}
  \begin{aligned}
    H(t)=\left(\begin{matrix}
      H_{11}(t)&H_{12}(t)&H_{13}(t)&H_{14}(t)\\
      H_{21}(t)&H_{22}(t)&H_{23}(t)&H_{24}(t)\\
      H_{31}(t)&H_{32}(t)&H_{33}(t)&0\\
      H_{41}(t)&H_{42}(t)&0&H_{44}(t)\\
    \end{matrix}\right),\quad
    \bar H(t)=\left(\begin{matrix}
      \bar H_{11}(t)&\bar H_{12}(t)&\bar H_{13}(t)&\bar H_{14}(t)\\
      \bar H_{21}(t)&\bar H_{22}(t)&\bar H_{23}(t)&\bar H_{24}(t)\\
      \bar H_{31}(t)&\bar H_{32}(t)&\bar H_{33}(t)&0\\
      \bar H_{41}(t)&\bar H_{42}(t)&0&\bar H_{44}(t)\\
    \end{matrix}\right).
  \end{aligned}
\end{equation}
Here, $H_{kl}(t), \bar H_{kl}(t)\in C(0,T;\mathbb R^{n\times n})$, $k,l=1,2,3,4$ such that $H_{kl}^{\top}(t)=H_{lk}(t)$. Thus for the FBSDE with time-dependent coefficients \eqref{coefficient} without perturbation, i.e., $\bar H(t)\equiv 0$, (iii) in Assumption (H1) (the monotonicity condition) is equivalent to, for uniformly $t\in [0,T]$,
\begin{equation}\label{mono}
  \begin{aligned}
    \left(\begin{matrix}
      -H_{11}(t)&-H_{12}(t)&-H_{13}(t)&-H_{14}(t)\\
      H_{21}(t)&H_{22}(t)&H_{23}(t)&H_{24}(t)\\
      H_{31}(t)&H_{32}(t)&H_{33}(t)&0\\
      H_{41}(t)&H_{42}(t)&0&H_{44}(t)\\
    \end{matrix}\right)\leq -\beta I_{4n},
  \end{aligned}
\end{equation}
where $\beta>0$ is a constant and $I_n$ is the $n$-order unit matrix. Furthermore, we also have
\begin{equation*}\begin{aligned}
  H_{11}(t)\geq \beta I_{n},\quad H_{22}(t)\leq -\beta I_n,\quad H_{33}(t)\leq -\beta I_n,\quad
  H_{44}(t)\leq -\beta I_n,\quad t\in [0,T],
\end{aligned}\end{equation*}
and
\begin{equation}\label{377}\begin{aligned}
  -H_{11}(t)+H_{13}(t)H_{33}^{-1}(t)H_{31}(t) +H_{14}(t)H_{44}^{-1}(t)H_{41}(t)&<0,\\
  H_{22}(t)-H_{23}(t)H^{-1}_{33}(t)H_{32}(t)
  -H_{24}(t)H_{44}^{-1}(t)H_{42}(t)&<0.
\end{aligned}\end{equation}
Then the stochastic Hamiltonian system \eqref{Hamiltonian} can be rewritten as
\begin{equation}\label{LFBSDE}
  \left\{
  \begin{aligned}
    &\mathrm{d}x_t=\{H_{21}^{\rho}(t)x_t+H_{22}^{\rho}(t)y_t +H_{23}^{\rho}(t)z_t+ H_{24}^{\rho}(t)\theta_{t}\}\mathrm{d}t\\
    &\qquad+[H_{31}^{\rho}(t)x_t+H_{32}^{\rho}(t)y_t+H_{33}^{\rho}(t)z_t]\mathrm{d}B_t\\
    &\qquad +[H_{41}^{\rho}(t)x_{t-}+ H_{42}^{\rho}(t)y_{t-}+H_{44}^{\rho}(t)\theta_{t}]\mathrm{d}\tilde V_t,\\
    &\mathrm{d}y_t=-[H_{11}^{\rho}(t)x_t+H_{12}^{\rho}(t)y_t+H_{13}^{\rho}(t)z_t+H_{14}^{\rho}(t)\theta_{t}]
    \mathrm{d}t+z_t\mathrm{d}B_t+\theta_{t}\mathrm{d}\tilde V_t,\\
    &x_0=0,\qquad y_T=0.
  \end{aligned}
  \right.
\end{equation}
Now we first give the definition of eigenvalues and eigenfunctions of stochastic Hamiltonian system with Markov chain as follows.
\begin{definition}
  A real number $\rho$ is called an eigenvalue of the linear stochastic Hamiltonian system \eqref{LFBSDE} if there exist nontrivial solutions $(x,y,z,\theta)\not\equiv 0$ of regime-switching FBSDE \eqref{LFBSDE} corresponding to $\rho$. These solutions are the eigenfunctions corresponding to $\rho$. All eigenfunctions corresponding to the eigenvalue $\rho$ constitute of a linear subspace of $L^2_{\mathbb F}(0,T;\mathbb R^n)\times L^2_{\mathbb F}(0,T;\mathbb R^n)\times F^2_{\mathbb F}(0,T;\mathbb R^n)\times M^2_{\mathbb F}(0,T;\mathbb R^{n})$ called the eigenfunction subspace corresponding to $\rho$.
\end{definition}

In fact, if the condition \eqref{mono} holds, regime-switching FBSDE \eqref{LFBSDE} admits a unique solution $(x,y,z,\theta)=(0,0,0,0)$ corresponding to $\rho=0$ by Theorem \ref{thm21}. Thus $\rho=0$ cannot be an eigenvalue of stochastic Hamiltonian system \eqref{LFBSDE}.

\section{Dual Transformation of Hamiltonian System}\label{3}
Inspired by \cite{P2000,WW2017}, we introduce the dual transformation, which is a powerful tool for solving the eigenvalue problem, of stochastic Hamiltonian system with Markov chain.

\subsection{General Case}

In this subsection, if the stochastic Hamiltonian system \eqref{Hamiltonian} admits a unique solution, we exchange the role of $x$ and $y$, i.e. $(\tilde x,\tilde y)=(y,x)$, to see whether $(\tilde x,\tilde y)$ still satisfy some stochastic Hamiltonian system. In order to study this problem, we need introduce the following assumption.

\noindent \textbf{Assumption (H2)} $h$ and $\Phi$ are both $\mathcal C^2$ functions. Moreover, for all $(x,y)\in  \mathbb R^n\times \mathbb R^n$, $h(t,x,y,\cdot,\cdot)$ is concave and $\Phi(\cdot)$ is convex.  

Then we give the following Legendre transformation for $h$ with respect to $(z,\theta)$, and for $\Phi$ with respect to $x$:
\begin{equation*}
  \begin{aligned}
    \tilde h(t,\tilde x,\tilde y,\tilde z,\tilde\theta)&=\inf_{(z,\theta)\in \mathbb R^{2n}}\big\{\langle z,\tilde z\rangle+\langle \theta,\tilde\theta\rangle-h(t,\tilde y,\tilde x,z,\theta)\big\}\\
    &=\langle z^*,\tilde z\rangle+\langle \theta_1^*,\tilde\theta_1\rangle+\langle \theta_2^*,\tilde\theta_2\rangle-h(t,\tilde y,\tilde x,z^*,\theta^*),\\
    \tilde\Phi(\tilde x)&=\sup_{x\in \mathbb R^n}\big\{\langle x,\tilde x\rangle -\Phi(x)\big\}=\langle x^*(\tilde x),\tilde x\rangle -\Phi(x^*(\tilde x)),
  \end{aligned}
\end{equation*}
where $(z^*,\theta^*):=(z^*(t,\tilde x,\tilde y,\tilde z,\tilde\theta),\theta^*(t,\tilde x,\tilde y,\tilde z,\tilde\theta))$ is the unique minimum point for each $(t,\tilde x,\tilde y,\tilde z,\tilde \theta)$, and $x^*(\tilde x)$ is the unique maximum point for each $\tilde x$. $\tilde h$ is called the dual Hamiltonian function of \eqref{Hamiltonian}. Under Assumption (H2), we can know that $\tilde h$ and $\tilde\Phi$ are well-defined. Moreover, we have
\begin{equation*}
  \begin{aligned}
    h(t,x,y,z,\theta)&=\inf_{(\tilde z,\tilde\theta)\in \mathbb R^{2n}}\big\{\langle z,\tilde z\rangle+\langle \theta,\tilde\theta\rangle-\tilde h(t,y,x,\tilde z,\tilde \theta)\big\}\\
    &=\langle z,\tilde z^*\rangle+\langle \theta_1,\tilde\theta_1^*\rangle+\langle \theta_2,\tilde\theta_2^*\rangle-\tilde h(t,y,x,\tilde z^*,\tilde\theta^*),\\
    \Phi(x)&=\sup_{\tilde x\in \mathbb R^n}\big\{\langle x,\tilde x\rangle-\tilde\Phi(\tilde x)\big\}=\langle x,\tilde x^*(x)\rangle-\tilde\Phi(\tilde x^*(x)),
  \end{aligned}
\end{equation*}
where $(\tilde z^*,\tilde\theta^*):=(\tilde z^*(t,x,y,z,\theta), \tilde\theta^*(t,x,y,z,\theta)$, is the unique minimum point for each $(t,x,y,z,\theta)$, and $\tilde x^*(x)$ is the unique maximum point for each $x$. Furthermore, we obtain
\begin{equation}
  \begin{aligned}
    &\tilde x=\partial_x\Phi(x),\quad &&\tilde z=\partial_z h(t,\tilde y,\tilde x,z,\theta),\\ &\tilde\theta=\partial_{\theta} h(t,\tilde y,\tilde x,z,\theta), \quad   &&\forall (\tilde x,\tilde y,\tilde z,\tilde\theta)\in \mathbb R^n\times \mathbb R^n\times \mathbb R^n\times \mathbb R^{n},\\
    &x=\partial_{\tilde x}\tilde\Phi(\tilde x), \quad &&z=\partial_{\tilde z} \tilde h(t,y,x,\tilde z,\tilde\theta),\\
    &\theta=\partial_{\tilde\theta} \tilde h(t,y,x,\tilde z,\tilde\theta),\quad  &&\forall (x,y,z,\theta)\in \mathbb R^n\times \mathbb R^n\times \mathbb R^n\times \mathbb R^{n}.
  \end{aligned}
\end{equation}
Then we can easily know that the quadruple $(\tilde x_t,\tilde y_t,\tilde z_t,\tilde \theta_t)$ defined by
\begin{equation}\label{12612}
  \begin{aligned}
    &\tilde x_t=\tilde y_t,\quad \tilde y_t=\tilde x_t,\quad \tilde z_t=\partial_zh(t,x_{t-},y_{t-},z_t,\theta_{t}),\quad
    \tilde \theta_{t}=\partial_{\theta}h(t,x_{t-},y_{t-},z_t, \theta_{t}),
  \end{aligned}
\end{equation}
satisfies the following stochastic Hamilton system with Markov chain,
\begin{equation}\label{dual}\left\{
  \begin{aligned}
    &\mathrm{d}\tilde x_t=\partial_{\tilde y}\tilde h(t,\tilde x_t,\tilde y_t,\tilde z_t,\tilde \theta_{t})\mathrm{d}t +\partial_{\tilde z}\tilde h(t,\tilde x_t,\tilde y_t,\tilde z_t,\tilde \theta_{t})\mathrm{d}B_t +\partial_{\tilde \theta}\tilde h(t,\tilde x_{t-},\tilde y_{t-},\tilde z_t,\tilde \theta_{t}) \mathrm{d}\tilde V_t,\\
    &\mathrm{d}\tilde y_t=\partial_{\tilde x}\tilde h(t,\tilde x_t,\tilde y_t,\tilde z_t,\tilde \theta_{t})\mathrm{d}t+\tilde z_t\mathrm{d}B_t+\tilde\theta_{t}\mathrm{d}\tilde V_t,\\
    &\tilde x_0=y_0,\qquad \tilde y_T=\partial_{\tilde x}\tilde\Phi(\tilde x_T).
  \end{aligned}
\right.\end{equation}
Hamiltonian system \eqref{dual} is called the dual Hamiltonian system of the original Hamiltonian system \eqref{Hamiltonian}. And the Hamiltonian system \eqref{Hamiltonian} is also the dual of \eqref{dual}.

\subsection{Linear Case}

In this subsection, we study the linear case. Consider the following linear Hamiltonian system with regime-switching,
\begin{equation}\label{12614}
  \left\{
  \begin{aligned}
    &\mathrm{d}x_t=[H_{21}(t)x_t+H_{22}(t)y_t +H_{23}(t) z_t+H_{24}(t) \theta_{t}]\mathrm{d}t\\
    &\qquad\quad +[H_{31}(t)x_t+H_{32}(t)y_t+H_{33}(t)z_t]\mathrm{d}B_t\\
    &\qquad\quad+[H_{41}(t)x_{t-}+H_{42}(t)y_{t-}+H_{44}(t)\theta_{t}]\mathrm{d}\tilde V_t,\\
    &\mathrm{d}y_t=-[H_{11}(t)x_t+H_{12}(t)y_t+H_{13}(t)z_t+H_{14}(t) \theta_{t}]\mathrm{d}t+z_t\mathrm{d}B_t+\theta_{t}\mathrm{d}\tilde V_t,\\
    &x_0=x_0,\qquad y_T=Mx_T.
  \end{aligned}
  \right.
\end{equation}
where $H_{kl}(t)\in C(0,T;\mathbb R^n)$ and $M$ is a $n\times n$ matrix such that $H^{\top}_{kl}(t)=H_{lk}(t)$, $k,l=1,2,3,4$, $M^{\top}=M$. Assumption (H2) means that
\begin{equation*}
  H_{33}(t)<0,\quad H_{44}(t)<0,\quad H_{55}(t)<0,\quad \text{and}\quad M>0,\quad  t\in [0,T].
\end{equation*}
According to \eqref{12612}, we define
\begin{equation*}
  \begin{aligned}
    &\tilde x_t= y_t,\qquad \tilde z_t=H_{31}(t)x_{t-}+H_{32}(t)y_{t-}+H_{33} (t)z_t,\\
    &\tilde y_t= x_t,\qquad \tilde \theta_{t}=H_{41}(t)x_{t-}+H_{42}(t)y_{t-}+H_{44} (t)\theta_{t},
  \end{aligned}
\end{equation*}
then the dual Hamiltonian system of \eqref{12614} is
\begin{equation}\label{3913}
  \left\{
  \begin{aligned}
    &\mathrm{d}\tilde x_t=[\tilde H_{21}(t)\tilde x_t+\tilde H_{22}(t)\tilde y_t+\tilde H_{23}(t) \tilde z_t+\tilde H_{24}(t)\tilde \theta_{t}]\mathrm{d}t\\
    &\qquad\quad+[\tilde H_{31}(t)\tilde x_t+\tilde H_{32}(t)\tilde y_t+\tilde H_{33}(t) \tilde z_t]\mathrm{d}B_t\\
    &\qquad\quad +[\tilde H_{41}(t)\tilde x_{t-}+\tilde H_{42}(t)\tilde y_{t-} +\tilde H_{44}(t)\tilde \theta_{t}]\mathrm{d}\tilde V_t,\\
    &\mathrm{d}\tilde y_t=-[\tilde H_{11}(t)\tilde x_t+\tilde H_{12}(t)\tilde y_t+\tilde H_{13}(t) \tilde z_t+\tilde H_{14}(t)\tilde \theta_{t}]\mathrm{d}t+\tilde z_t\mathrm{d}B_t+\tilde\theta_{t}\mathrm{d}\tilde V_t,\\
    &\tilde x_0=y_0,\qquad \tilde y_T=M^{-1}\tilde x_T,
  \end{aligned}
  \right.
\end{equation}
where (we suppress the variance $t$)
\begin{equation*}
  \begin{aligned}
    \tilde H=\left(\begin{matrix}
      \tilde H_{11}&\tilde H_{12}& -H_{23}H_{33}^{-1}&-H_{24}H_{44}^{-1}\\
      \tilde H_{21}&\tilde H_{22}&-H_{13}H_{33}^{-1}&-H_{14}H_{44}^{-1}\\
      -H_{33}^{-1}H_{32}&-H_{33}^{-1}H_{31}&H_{33}^{-1}&0\\
      -H_{44}^{-1}H_{42}&-H_{44}^{-1}H_{41}&0&H_{44}^{-1}
    \end{matrix}\right).
  \end{aligned}
\end{equation*}
with $\tilde H_{11}=H_{23}H_{33}^{-1}H_{32}+H_{24}H_{44}^{-1}H_{42}-H_{22}$, $\tilde H_{12}=H_{23}H_{33}^{-1}H_{31} +H_{24}H_{44}^{-1}H_{41}-H_{21}$, $\tilde H_{21}=H_{13}H_{33}^{-1}H_{32}+H_{14}H_{44}^{-1}H_{42}-H_{12}$, $\tilde H_{22} =H_{13}H_{33}^{-1}H_{31}+H_{14}H_{44}^{-1}H_{41}-H_{11}$.

\subsection{Generalized Riccati equation systems}\label{3.3}
In this subsection, we would like to use the generalized Riccati equation systems to analyze the stochastic Hamiltonian system \eqref{12614}. The Riccati equation is widely applied to study the linear-quadratic optimal control problems, see \cite{W1968,ZES2012,ZLX2021}. Inspired by \cite{ZES2012}, we introduce a dynamic system on some interval $[T_1,T_2]\subseteq [0,T]$ as follows,
\begin{equation}\label{Riccati1}
  \left\{
  \begin{aligned}
    &-\dot{K}_t=K_t[H_{21}(t)+H_{22}(t)K_t+H_{23}(t)L_t+H_{24}(t)P_t]\\
    &\qquad\qquad\quad +H_{11}(t)+H_{12}(t)K_t+H_{13}(t)L_t+H_{14}(t)P_t,\\
    &L_t=K_t[H_{31}(t)+H_{32}(t)K_t+H_{33}(t)L_t],\\
    &P_t=K_t[H_{41}(t)+H_{42}(t)K_t+H_{44}(t)P_t],\\
    &K_{T_2}=\bar K\in \mathbb S^n.
  \end{aligned}
  \right.
\end{equation}
This system is called a generalized Riccati equation system. Note that the three algebraic equations in \eqref{Riccati1} are equivalent to
\begin{equation*}
  \begin{aligned}
    &[I_n-K_tH_{33}(t)]L_t=K(t)[H_{31}(t)+H_{32}(t)K_t],\\
    &[I_n-K_tH_{44}(t)]P_t=K(t)[H_{41}(t)+H_{42}(t)K_t],
  \end{aligned}
\end{equation*}
If we assume that
\begin{equation}\label{assum}
  \begin{aligned}
    &\det[I_n-K_tH_{33}(t)]\ne 0,\qquad \det[I_n-K_tH_{44}(t)]\ne 0,
  \end{aligned}
\end{equation}
Hence $[I_n-K_tH_{33}(t)]^{-1}$ and $[I_n-K_tH_{44}(t)]^{-1}$ exist and are both uniformly bounded due to the continuity of $\det[I_n-K_tH_{33}(t)]$ and $\det[I_n-K_tH_{44}(t)]$. Then $L_{\cdot}$ and $P_{\cdot}$ can be represented by $K_{\cdot}$ as follows,
\begin{equation*}
  \begin{aligned}
    &L_t=F_0(K_t)[H_{31}(t)+H_{32}(t)K_t],\\
    &P_t=F_1(K_t)[H_{41}(t)+H_{42}(t)K_t],
  \end{aligned}
\end{equation*}
where $F_0(K_t)=[I_n-K_tH_{33}(t)]^{-1}K_t$ and $F_1(K_t)=[I_n-K_tH_{44}(t)]^{-1}K_t$. Thus \eqref{Riccati1} can be rewritten as
\begin{equation}\label{3616}
  \begin{aligned}
    -\dot{K}_t&=K_tH_{21}(t)+H_{12}(t)K_t+H_{11}(t)+K_tH_{22}(t)K_t\\
    &\qquad+[K_tH_{23}(t)+H_{13}(t)]F_0(K_t)[H_{31}(t)+H_{32}(t) K_t]\\
    &\qquad+[K_tH_{24}(t)+H_{14}(t)]F_1(K_t)[H_{41}(t)+H_{42}(t) K_t].
  \end{aligned}
\end{equation}
\begin{remark}\label{rem31} Let $\delta_1>\delta>0$ such that $-\delta_1 I_n\leq H_{33}(t)\leq -\delta I_n$ and $-\delta_1 I_n\leq H_{44}(t)\leq -\delta I_n$, $t\in [0,T]$, then we have

  i) For any $K\in \mathbb S^n$ satisfying $K>-\frac{1}{2\delta_1}I_n$, there exists a positive constant $c$ such that
  \begin{equation*}
    \|(I_n-KH_{33}(t))^{-1}\|\leq c,\qquad \|(I_n-KH_{44}(t))^{-1}\|\leq c,\quad \forall t\in [0,T].
  \end{equation*}

  ii) For any $K\in \mathbb S^n_+$, $(F_0(K))^{\top}=F_0(K)$ and $(F_1(K))^{\top}=F_1(K)$. If $K>0$, then 
  \begin{equation*}
    0\leq F_0(K)\leq -H_{33}^{-1}(t)\leq -\frac{1}{\delta} I_n,\qquad 0\leq F_1(K)\leq -H_{44}^{-1}(t)\leq -\frac{1}{\delta} I_n.
  \end{equation*}

  iii) For any $K_1,K_2\in \mathbb S^n_{+}$, $K_1\geq K_2$, then $F_0(K_1)\geq F_0(K_2)$ and $F_1(K_1)\geq F_1(K_2)$.
\end{remark}

Next, we will show the relationship between the generalized Riccati equation system \eqref{Riccati1} and linear Hamiltonian system \eqref{12614}.

\begin{lemma}\label{lemma3.2}
  Suppose that \eqref{Riccati1} admits a unique solution $(K_{\cdot},L_{\cdot},P_{\cdot})$ on some interval $[T_1,T_2]\subseteq [0,T]$. Then Hamiltonian system \eqref{12614} with the bounded condition
  \begin{equation}\label{22116}
    x_{T_1}=x_0,\qquad y_{T_2}=\bar K x_{T_2},
  \end{equation}
  admits a solution
  \begin{equation}\label{22117}
    (x_t,y_t,z_t,\theta_{t})=(X_t,K_tX_t,L_tX_{t-},P_tX_{t-}),
  \end{equation}
  where $X_t$ satisfies
  \begin{equation}\label{22519}
  \left\{
  \begin{aligned}
    &\mathrm{d}X_t=[H_{21}(t)+H_{22}(t)K_t+H_{23}(t)L_t+H_{24}(t) P_t]X_t\mathrm{d}t\\
    &\qquad\quad +[H_{31}(t)+H_{32}(t)K_t +H_{33}(t)L_t]X_t\mathrm{d}B_t\\
    &\qquad\quad+[H_{41}(t)+H_{42}(t)K_t+H_{44}(t)P_t]X_{t-} \mathrm{d}\tilde V_t,\\
    &X_{T_1}=x_0.
  \end{aligned}
  \right.
\end{equation}
Moreover, let condition \eqref{assum} or the following weaker condition hold,
\begin{equation}\label{weaker}
  \begin{aligned}
    &[I_n-K_tH_{33}(t)]^{\top}[I_n-K_tH_{33}(t)]\geq c[H_{13}(t)+K_tH_{23}(t)]^{\top}[H_{13}(t)+K_tH_{23}(t)],\\
    &[I_n-K_tH_{44}(t)]^{\top}[I_n-K_tH_{44}(t)]\geq c[H_{14}(t)+K_tH_{24}(t)]^{\top}[H_{14}(t)+K_tH_{24}(t)],
  \end{aligned}
\end{equation}
where $c$ is a positive constant, then the Hamiltonian system \eqref{12614} with boundary condition \eqref{22116} admits a unique solution.
\end{lemma}
\begin{proof}
  Applying It\^o's formula to $K_tx_t$, we easily verify that \eqref{22117} is a solution of the Hamiltonian system \eqref{12614} with boundary condition \eqref{22116}.

  Then we consider the uniqueness under the condition \eqref{assum}. Let $(x_t,y_t,z_t,\theta_{t})$ be a solution of \eqref{12614} with boundary condition \eqref{22116} and $(\bar y_t,\bar z_t,\bar \theta_{t})=(K_tx_t,L_tx_{t-}, P_tx_{t-})$. Then applying It\^o's formula to $\bar y_t$, we obtain
  \begin{equation*}
    \begin{aligned}
      -\mathrm{d}\bar y_t
      &=\{K_t[H_{22}(t)\bar y_t+H_{23}(t)\bar z_t+H_{24}(t)\bar\theta_{t} ]+H_{11}(t)x_t+H_{12}(t)\bar y_t +H_{13}(t)\bar z_t+H_{14}(t)\bar \theta_{t}\}\mathrm{d}t\\
      &\quad -K_t[H_{22}(t)y_t+H_{23}(t) z_t +H_{24}(t)\theta_{t}]\mathrm{d}t -K_t[H_{31}(t)x_t+H_{32}(t)y_t+H_{33}(t)z_t]\mathrm{d}B_t\\
      &\quad-K_t[H_{41}(t)x_{t-}+H_{42}(t) y_{t-}+H_{44}(t)\theta_{t}]\mathrm{d}\tilde V_t.
    \end{aligned}
  \end{equation*}
  Define $(\hat y_t,\hat z_t,\hat\theta_{1,t},\hat\theta_{2,t})=(\bar y_t-y_t,\bar z_t-z_t,\bar\theta_{1,t}-\theta_{1, t},\bar\theta_{2,t}-\theta_{2,t})$, thus
  \begin{equation*}
    \begin{aligned}
      -\mathrm{d}\bar y_t
      &=\{K_t[H_{22}(t)\hat y_t+H_{23}(t)\hat z_t+H_{24}(t)\hat\theta_{t}]+H_{11}(t)x_t+H_{12}(t)\bar y_t+H_{13}(t)\bar z_t\\
      &\quad+H_{14}(t)\bar \theta_{t}\}\mathrm{d}t -K_t[H_{31}(t)x_t+H_{32}(t)y_t+H_{33}(t)z_t]\mathrm{d}B_t\\
      &\quad-K_t[H_{41}(t)x_{t-}+H_{42}(t) y_{t-}+H_{44}(t)\theta_{t}]\mathrm{d}\tilde V_t.
    \end{aligned}
  \end{equation*}
  Noticing that
  \begin{equation*}
    \begin{aligned}
      &L_t=K_t[H_{31}(t)+H_{32}(t)K_t+H_{33}(t)L_t],\\
    &P_t=K_t[H_{41}(t)+H_{42}(t)K_t+H_{44}(t)P_t],
    \end{aligned}
  \end{equation*}
  then
  \begin{equation*}
    \begin{aligned}
      &K_tH_{31}(t)x_{t-}=\bar z_t-K_t[H_{32}(t)\bar y_{t-}+H_{33}(t)\bar z_{t}],\\
      &K_tH_{41}(t)x_{t-}=\bar\theta_{t}-K_t[H_{42}(t)\bar y_{t-}+H_{44}(t)\bar \theta_{t}].
    \end{aligned}
  \end{equation*}
  Hence, we obtain
  \begin{equation*}
    \begin{aligned}
      -\mathrm{d}\hat y_t
      &=\{K_t[H_{22}(t)\hat y_t+H_{23}(t)\hat z_t+H_{24}(t)\hat\theta_{t}] +H_{12}(t)\hat y_t+H_{13}(t)\hat z_t+H_{14}(t)\hat \theta_{t}\}\mathrm{d}t\\
      &\quad+\{K_t[H_{32}(t)\hat y_t+H_{33}(t)\hat z_t]-\hat z_t\}\mathrm{d}B_t +\{K_t[H_{42}(t)\hat y_{t-}+H_{44}(t)\hat\theta_{t}]-\hat\theta_{t}\}\mathrm{d}\tilde V_t.
    \end{aligned}
  \end{equation*}
  Recalling condition \eqref{assum}, we can know that $[I_n-K_tH_{33}(t)]^{-1}$ and $[I_n-K_tH_{44}(t)]^{-1}$ are uniformly bounded, then above equation can be rewritten as follows,
  \begin{equation}\label{22520}
    \begin{aligned}
      -\mathrm{d}\hat y_t
      &=\{[H_{12}(t)+K_tH_{22}(t)]\hat y_t\\
      &\quad+[H_{13}(t)+K_tH_{23}(t)][I_n-K_tH_{33}(t)]^{-1}[\check z_t+K_tH_{32}(t)\hat y_t]\\
      &\quad+[H_{14}(t)+K_tH_{24}(t)][I_n-K_tH_{44}(t)]^{-1}[\check\theta_{t}
      +K_tH_{42}(t)\hat y_t]  \}\mathrm{d}t\\
      &\quad-\check z_t\mathrm{d}B_t-\check\theta_{1,t}\mathrm{d}\tilde V^1_t-\check \theta_{2,t}\mathrm{d}\tilde V^2_t.
    \end{aligned}
  \end{equation}
  where
  \begin{equation}\label{22521}
    \begin{aligned}
      &\check z_t=-K_tH_{32}(t)\hat y_t+[I_n-K_tH_{33}(t)]\hat z_t,\\
      &\check \theta_{t}=-K_tH_{42}(t)\hat y_t+[I_n-K_tH_{44}(t)]\hat \theta_{t}.
    \end{aligned}
  \end{equation}
  According to \cite[Theorem 2.1]{W1999}, \eqref{22520} admits a unique solution $(\hat y_t,\check z_t, \check\theta_{t})\equiv (0,0,0)$. Furthermore, we have $(\hat y_t,\hat z_t, \hat \theta_{t})\equiv (0,0,0)$ by the transformation \eqref{22521}, which implies that
  \begin{equation*}
    y_t=K_tx_t,\qquad z_t=L_tx_{t-},\qquad \theta_{t}=P_tx_{t-}.
  \end{equation*}

  In addition, under the weaker condition \eqref{weaker} rather than condition \eqref{assum}, we also can show that $(\hat y_t,\check z_t, \check\theta_{t})\equiv (0,0,0)$ by the similar method. Thus, it follows from the SDE in stochastic Hamiltonian system \eqref{12614} that $x_t$ is a solution to SDE \eqref{22519}. Hence we have $(x_t,y_t,z_t,\theta_{t})=(X_t,K_tX_t,L_t X_{t-},P_tX_{t-})$. Then the proof is completed.
\end{proof}

From the similar argument, we obtain the following remark of dual Hamiltonian system.
\begin{remark}
  The generalized Riccati system related to the dual Hamiltonian system has the following form,
  \begin{equation}\label{22723}
  \left\{
  \begin{aligned}
    &-\dot{\tilde K}_t=\tilde K_t[\tilde H_{21}(t)+\tilde H_{22}(t)\tilde K_t+\tilde H_{23}(t)\tilde L_t+\tilde H_{24}(t)\tilde P_t]\\
    &\qquad\qquad\quad +\tilde H_{11}(t)+\tilde H_{12}(t)\tilde K_t+\tilde H_{13}(t)\tilde L_t+\tilde H_{14}(t)\tilde P_t,\\
    &\tilde L_t=\tilde K_t[\tilde H_{31}(t)+\tilde H_{32}(t)\tilde K_t+\tilde H_{33}(t)\tilde L_t],\\
    &\tilde P_t=\tilde K_t[\tilde H_{41}(t)+\tilde H_{42}(t)\tilde K_t+\tilde H_{44}(t)\tilde P_t].
  \end{aligned}
  \right.
\end{equation}
If the following condition hold,
\begin{equation*}
  \begin{aligned}
    &\det[I_n-\tilde K_t\tilde H_{33}(t)]\ne 0,\qquad \det[I_n-\tilde K_t\tilde H_{44}(t)]\ne 0,
  \end{aligned}
\end{equation*}
then generalized Riccati equation system \eqref{22723} can be rewritten as
 \begin{equation}\label{3323}
  \begin{aligned}
    -\dot{\tilde K}_t&=\tilde K_t\tilde H_{21}(t)+\tilde H_{12}(t)\tilde K_t+\tilde H_{11}(t)+\tilde K_t\tilde H_{22}(t)\tilde K_t\\
    &\qquad+[\tilde K_t\tilde H_{23}(t)+\tilde H_{13}(t)]\tilde F_0(\tilde K_t)[\tilde H_{31}(t)+\tilde H_{32}(t) \tilde K_t]\\
    &\qquad+[\tilde K_t\tilde H_{24}(t)+\tilde H_{14}(t)]\tilde F_1(\tilde K_t)[\tilde H_{41}(t)+\tilde H_{42}(t) \tilde K_t].
  \end{aligned}
\end{equation}
where $\tilde F_0(\tilde K_t)=[I_n-\tilde K_t\tilde H_{33}(t)]^{-1}\tilde K_t$ and $\tilde F_1(\tilde K_t)=[I_n-\tilde K_t\tilde H_{44}(t)]^{-1}\tilde K_t$.
\end{remark}


\section{Multi-dimensional Case}
In this section, we investigate the eigenvalue problem of stochastic Hamiltonian system \eqref{LFBSDE} with the perturbation matrix $\bar H$ as follows,
\begin{equation*}
  \bar H(t)=\left(\begin{matrix}
    0&0&H_{13}(t)&H_{14}(t)\\
    0&H_{22}(t)&H_{23}(t)&H_{24}(t)\\
    H_{31}(t)&H_{32}(t)&0&0\\
    H_{41}(t)&H_{42}(t)&0&0
  \end{matrix}\right).
\end{equation*}
Then Hamiltonian system \eqref{LFBSDE} becomes
\begin{equation}\label{3825}
  \left\{
  \begin{aligned}
    &\mathrm{d}x_t=[H_{21}(t)x_t+\varrho H_{22}(t)y_t +\varrho H_{23}(t) z_t+\varrho H_{24}(t) \theta_{t}]\mathrm{d}t\\
    &\qquad\quad +[\varrho H_{31}(t)x_t+\varrho H_{32}(t)y_t+H_{33}(t)z_t]\mathrm{d}B_t\\
    &\qquad\quad+[\varrho H_{41}(t)x_{t-}+\varrho H_{42}(t)y_{t-}+H_{44}(t)\theta_{t}]\mathrm{d}\tilde V_t,\\
    &\mathrm{d}y_t=-[H_{11}(t)x_t+H_{12}(t)y_t+\varrho H_{13}(t)z_t+\varrho H_{14}(t) \theta_{t}]\mathrm{d}t+z_t\mathrm{d}B_t+\theta_{t}\mathrm{d}\tilde V_t,\\
    &x_0=0,\qquad y_T=0,
  \end{aligned}
  \right.
\end{equation}
where $\varrho =1-\rho$. And we introduce some assumption to ensure the stochastic Hamiltonian system is well-defined.

\noindent \textbf{Assumption (H3)}

i) $H^{\top}=H$, $H_{kl}\in C(0,T;\mathbb R^{n\times n})$, $k,l=1,2,3,4$.

ii) $H(t)$ satisfies the condition \eqref{mono} uniformly in $t\in [0,T]$.

Then the corresponding Riccati equation \eqref{3616} can be rewritten as
\begin{equation}\label{3725}\left\{
  \begin{aligned}
    &-\dot{K}_t=K_tH_{21}(t)+H_{12}(t)K_t+H_{11}(t)+\varrho K_tH_{22}(t)K_t\\
    &\qquad\quad+\varrho^2 [K_tH_{23}(t)+H_{13}(t)]F_0(K_t)[H_{31}(t)+H_{32}(t) K_t]\\
    &\qquad\quad+\varrho^2[K_tH_{24}(t)+H_{14}(t)]F_1(K_t)[H_{41}(t)+H_{42}(t) K_t],\\
    &K_T=0.
  \end{aligned}
\right.\end{equation}
From the classical ODE theory, we know that Riccati equation \eqref{3725} admits a unique solution on a sufficiently small interval $(t_{\varrho},T]$. It follows from \cite[Lemma 8.1]{P2000} that $ K_t(\varrho)\geq 0$. Here, $t^K_{\varrho}$ is the so-called `blow-up time' of Riccati equation \eqref{3725}. More specifically, we give the definition of the blow-up time as follows,
\begin{equation*}
  t_{\varrho}^K:=\sup\big\{t_0|t_0<T, \lim_{t\searrow t_0}\|K_t(\varrho)\|=\infty\big\}.
\end{equation*}

Then we have the following result about the property of Blow-up time.

\begin{lemma}\label{lemma41}
  Let Assumption (H3) hold. For Riccati equation \eqref{3725}, when $\varrho\in [0,1]$, there is no explosion occurring, i.e. $t^K_{\varrho}=-\infty$. When $\varrho\in (-\infty,0)$, the blow-up time $t^K_{\varrho}<T$ is a finite number that is decreasing with respect to $\varrho$. Moreover, we obtain
  \begin{equation}\label{3226}
    \lim_{\varrho\searrow -\infty}t^K_{\varrho}=T,\qquad \lim_{\varrho\nearrow 0}t^K_{\varrho}=-\infty.
  \end{equation}
\end{lemma}
\begin{proof}
  For the case $\varrho\in [0,1]$. Noticing the inequality \eqref{377} and Remark \ref{rem31}, we obtain
  \begin{equation*}
    \begin{aligned}
      &\varrho H_{22}(t)+\varrho^2 H_{23}(t)F_0(K_t)H_{32}(t)+\varrho^2 H_{24}(t)F_0(K_t)H_{42}(t)\\
      &\leq \varrho H_{22}(t)+\varrho^2 H_{23}(t)(-H_{33}^{-1}(t))H_{32}(t)+\varrho^2 H_{24}(t)(-H_{44}^{-1}(t))H_{42}(t)\\
      &\leq 0.
    \end{aligned}
  \end{equation*}
  Then we introduce the following one-dimensional Riccati equation,
  \begin{equation}\label{3728}
    \left\{
    \begin{aligned}
      &-\dot k_{1,t}=2\Big(\|H_{12}\|_{\infty}+\frac{\varrho^2}{\delta}\|H_{13}\|_{\infty}\|H_{23}\|_{\infty} +\frac{\varrho^2}{\delta}\|H_{14}\|_{\infty}\|H_{24}\|_{\infty}\Big)k_{1,t}\\
      &\qquad\qquad +\|H_{11}\|_{\infty}+\frac{\varrho^2}{\delta}\|H_{13}\|^2_{\infty} +\frac{\varrho^2}{\delta}\|H_{14}\|^2_{\infty},\\
      &k_{1,T}=0.
    \end{aligned}
    \right.
  \end{equation}
  It is easy to see that Riccati equation \eqref{3728} admits a unique solution $K_{1,t}(\varrho)\geq 0$. In fact,
  \begin{equation*}
    \begin{aligned}
      k_{1,t}(\varrho)&=\frac{\|H_{11}\|_{\infty}+\frac{\varrho^2}{\delta}\|H_{13}\|^2_{\infty} +\frac{\varrho^2}{\delta}\|H_{14}\|^2_{\infty}}{2 (\|H_{12}\|_{\infty}+\frac{\varrho^2}{\delta}\|H_{13}\|_{\infty}\|H_{23}\|_{\infty} +\frac{\varrho^2}{\delta}\|H_{14}\|_{\infty}\|H_{24}\|_{\infty} )}\\
      &\times \bigg[\exp\Big\{2\Big(\|H_{12}\|_{\infty}+\frac{\varrho^2}{\delta}\|H_{13}\|_{\infty}\|H_{23}\|_{\infty} +\frac{\varrho^2}{\delta}\|H_{14}\|_{\infty}\|H_{24}\|_{\infty}\Big)(T-t)\Big\}-1\bigg],
    \end{aligned}
  \end{equation*}
  which implies that the blow-up time of $k_{1}$ is $t^{k_1}_{\varrho }=-\infty$. We also can obtain the same conclusion of $K_1=k_1I_n$. By virtue of the comparison theorem (\cite[Lemma 8.2]{P2000}), we have $0\leq K_t(\varrho)\leq K_{1,t}(\varrho)$, $t\in [0,T]$. Hence, the first desired result can be obtained.

  For the case $\varrho<0$. Let $K_{2,\cdot}(\varrho)$ be the solution to the following Riccati equation
  \begin{equation*}
    \left\{
    \begin{aligned}
      &-\dot{K}_{2,t}=K_{2,t}H_{21}(t)+H_{12}(t)K_{2,t}+H_{11}(t)+\varrho K_{2,t}H_{22}(t)K_{2,t},\\
      &K_{2,T}=0.
    \end{aligned}
    \right.
  \end{equation*}
  Since $F_0(K_t)\geq 0$, $F_1(K_t)\geq 0$, $H_{11}(t)\geq \beta I_n$ and $H_{22}(t)<-\beta I_n$, $t\in [0,T]$, we obtain $K_t(\varrho)\geq K_{1,t}(\varrho)$ by comparison theorem. If we can prove that the blow-up time of $K_{2,t}(\varrho)$ satisfies $\lim_{\varrho\searrow -\infty}t_{\varrho}^{K_2}=T$, we directly obtain the desired result. In fact, the blow-up $t_{\varrho}^{K_2}$ as $\varrho\searrow -\infty$ is same as that of $t_{\varrho}^{K_3}$, which is the blow-up time of $K_{3,t}(\varrho)=e^{\int_0^tH_{12}(s)\mathrm{d}s}K_{2,t}e^{ \int_0^t H_{21} (s)\mathrm{d}s}$. By some calculation, we have
  \begin{equation*}
    \left\{
    \begin{aligned}
      &-\dot{K}_{3,t}=e^{\int_0^tH_{12}(s)\mathrm{d}s}H_{11}(t)e^{\int_0^tH_{21}(s)\mathrm{d}s}+\varrho K_{3,t}e^{-\int_0^tH_{12}(s)\mathrm{d}s}H_{22}(t)e^{-\int_0^tH_{21}(s)\mathrm{d}s}K_{3,t},\\
      &K_{3,T}=0.
    \end{aligned}
    \right.
  \end{equation*}
  Moreover, there exist two constants $\beta_2,\beta_3>0$, such that, for all $t\in [0,T]$, $e^{\int_0^tH_{12}(s)\mathrm{d}s}H_{11}(t)e^{\int_0^tH_{21}(s)\mathrm{d}s}\geq \beta_2I_n$ and $e^{-\int_0^tH_{12}(s)\mathrm{d}s}H_{22}(t)e^{-\int_0^tH_{21}(s)\mathrm{d}s}\geq -\beta_3 I_n$. Let $-\dot{k}_{4,t}=\beta_2-\varrho\beta_3 K_{4,t}^2$ with $k_{t,T}=0$, whose blow-up time $t_{\varrho}^{k_4}$ satisfies $\lim_{\varrho\searrow -\infty}t_{\varrho}^{k_4}=T$. By virtue of comparison theorem, we get $K_{3,t}\geq K_{4,t}=k_{4,t}I_n$. Then we obtain the first limitation.

  In addition, let $K_{5,t}$ be the solution of
  \begin{equation}\label{3729}
    \left\{
    \begin{aligned}
      &-\dot{K}_{5,t}=K_{5,t}H_{21}(t)+H_{12}(t)K_{5,t}+H_{11}(t)+\varrho K_{5,t}H_{22}(t)K_{5,t}\\
      &\qquad\qquad -\varrho^2[H_{13}(t)+K_{5,t}H_{23}(t)]H_{33}^{-1}(t)[H_{32}(t)K_{5,t}+H_{31}(t)]\\
      &\qquad\qquad-\varrho^2[H_{14}(t)+K_{5,t}H_{24}(t)]H_{44}^{-1}(t)[H_{42}(t)K_{5,t}+H_{41}(t)],\\
      &K_{5,T}=0.
    \end{aligned}
    \right.
  \end{equation}
  By Remark \ref{rem31} and comparison theorem, we have $K_{t}(\varrho)\leq K_{5,t}(\varrho)$. In fact, when $\varrho=0$, Riccati equation \eqref{3729} becomes a linear ODE and the blow-up time $t_{\varrho}^{K_5}=-\infty$. Since the left-hand side of ODE \eqref{3729} is locally Lipschitz, it follows that, for an arbitrarily large interval $(\tilde t,T]$, $\tilde t\in (-\infty,T)$, there exists a sufficiently small $\epsilon>0 $, such that $K_{5,t}(\varrho)$, is uniformly bounded in $[\tilde t,T]$, for $\varrho\in [-\epsilon,\epsilon]$, and that $\lim_{\varrho\nearrow 0}\sup_{t\in [\tilde t, T]}\|K_{t}(\varrho)-K_t(0)\|=0$. Hence, the blow-up time $t_{\varrho}^{K_5}$ satisfies $\lim_{\varrho\nearrow 0}t_{\varrho}^{K_5}=T$, which implies the second limitation in \eqref{3226}.

  Applying comparison theorem again, we obtain
  \begin{equation*}
    0\leq K_t(\varrho_1)\leq K_{t}(\varrho_2),\qquad -\infty<\varrho_2<\varrho_1<0,\qquad t\in (t_{\varrho_2},T].
  \end{equation*}
  Thus $t_{\varrho}$ is a decreasing function of $\varrho$. Then the proof is complete.
\end{proof}
Recalling the dual Hamiltonian system \eqref{3913}, which can be rewritten as follows.
\begin{equation}
  \left\{
  \begin{aligned}
    &\mathrm{d}\tilde x_t=[\tilde H_{21}(t)\tilde x_t+\tilde H_{22}(t)\tilde y_t+\tilde H_{23}(t) \tilde z_t+\tilde H_{24}(t)\tilde \theta_{t}]\mathrm{d}t\\
    &\qquad\quad+[\tilde H_{31}(t)\tilde x_t+\tilde H_{32}(t)\tilde y_t+\tilde H_{33}(t) \tilde z_t]\mathrm{d}B_t\\
    &\qquad\quad +[\tilde H_{41}(t)\tilde x_{t-}+\tilde H_{42}(t)\tilde y_{t-} +\tilde H_{44}(t)\tilde \theta_{t}]\mathrm{d}\tilde V_t,\\
    &\mathrm{d}\tilde y_t=-[\tilde H_{11}(t)\tilde x_t+\tilde H_{12}(t)\tilde y_t+\tilde H_{13}(t) \tilde z_t+\tilde H_{14}(t)\tilde \theta_{t}]\mathrm{d}t+\tilde z_t\mathrm{d}B_t+\tilde\theta_{t}\mathrm{d}\tilde V_t,\\
    &\tilde x_0=0,\qquad \tilde y_T=0,
  \end{aligned}
  \right.
\end{equation}
where
\begin{equation}\label{3831}
  \begin{aligned}
    \tilde H=\left(\begin{matrix}
      \tilde H_{11}&\tilde H_{12}& -\varrho H_{23}H_{33}^{-1}&-\varrho H_{24}H_{44}^{-1}\\
      \tilde H_{21}&\tilde H_{22}&-\varrho H_{13}H_{33}^{-1}&-\varrho H_{14}H_{44}^{-1}\\
      -\varrho H_{33}^{-1}H_{32}&-\varrho H_{33}^{-1}H_{31}&H_{33}^{-1}&0\\
      -\varrho H_{44}^{-1}H_{42}&-\varrho H_{44}^{-1}H_{41}&0&H_{44}^{-1}
    \end{matrix}\right).
  \end{aligned}
\end{equation}
with $\tilde H_{11}=\varrho^2H_{23}H_{33}^{-1}H_{32}+\varrho^2H_{24}H_{44}^{-1}H_{42}-\varrho H_{22}$, $\tilde H_{12}=\varrho^2H_{23}H_{33}^{-1}H_{31} +\varrho^2H_{24}H_{44}^{-1}H_{41}-H_{21}$, $\tilde H_{21}=\varrho^2H_{13}H_{33}^{-1}H_{32}+\varrho^2H_{14}H_{44}^{-1}H_{42}-H_{12}$, $\tilde H_{22} =\varrho^2H_{13}H_{33}^{-1}H_{31}+\varrho^2H_{14}H_{44}^{-1}H_{41}-H_{11}$. Consequently, the dual Riccati equation \eqref{3323} becomes
 \begin{equation}\label{3830}
  \begin{aligned}
    -\dot{\tilde K}_t&=\tilde K_t\tilde H_{21}(t)+\tilde H_{12}(t)\tilde K_t+\tilde H_{11}(t)+\tilde K_t\tilde H_{22}(t)\tilde K_t\\
    &\qquad+[\tilde K_t\tilde H_{23}(t)+\tilde H_{13}(t)]\tilde F_0(\tilde K_t)[\tilde H_{31}(t)+\tilde H_{32}(t) \tilde K_t]\\
    &\qquad+[\tilde K_t\tilde H_{24}(t)+\tilde H_{14}(t)]\tilde F_1(\tilde K_t)[\tilde H_{41}(t)+\tilde H_{42}(t) \tilde K_t].
  \end{aligned}
\end{equation}
We define the blow-up time of $\tilde K_t$ as follows.
\begin{equation*}
   t^{\tilde K}_{\varrho}=\sup\big\{t_0|t_0<T,\lim_{t\searrow t_0}\|\tilde K_{t}(\varrho)\|=\infty\big\}.
\end{equation*}
From the Legendre transformation in Section \ref{3}, we know $y_t=K_tx_t$, $\tilde y_t=\tilde K_t\tilde x_t$ and $(\tilde x,\tilde y)=(y,x)$. Furthermore, if the inverse of $K$ exists, then $\tilde K=K^{-1}$. Hence, we obtain the following relationship of the blow-up time between $K$ and $\tilde K$.
\begin{equation*}\begin{aligned}
  &t^K_{\varrho}=\sup\big\{t_0|t_0<T, \lim_{t\searrow t_0}\|\tilde K_t(\varrho)\|=0\big\},\\
  & t^{\tilde K}_{\varrho}=\sup\big\{t_0|t_0<T, \lim_{t\searrow t_0}\| K_t(\varrho)\|=0\big\}.
\end{aligned}\end{equation*}

Base on Lemma \ref{lemma41}, we have the following results by a similar method of \cite[Lemma 5.2-Lemma 5.3]{P2000}.
\begin{lemma}\label{lemma42}
  Under Assumption (H3), the blow-up time $t_{\varrho}$ depends continuously with respect to $\varrho\in (-\infty,0)$, which is a strictly decreasing function of $\varrho$.
\end{lemma}

\begin{lemma}\label{lemma43}
   Suppose that Assumption (H3) holds. The solution of Riccati equation \eqref{3725} is continuous on $t\in (t_{\varrho}^K,T]$. Its inverse $\tilde K_t=K_t^{-1}$ is continuous on $t\in [t_{\varrho}^K,T)$ and satisfies the dual Riccati equation \eqref{3830}. Moreover, the matrix defined by $\check K_0=\lim_{t\searrow t_{\varrho}^K}K^{-1}_t=\tilde K_{t_{\varrho}}\geq 0$ is degenerate.
\end{lemma}

Next we present one of the main results of this paper.
\begin{theorem}\label{theorem44}
  Let Assumption (H3) hold. Then there exists a $\rho_1>0$, the smallest eigenvalue. The dimension of the space of the eigenfunctions corresponding to $\rho_1$ is less than $n$.
\end{theorem}
\begin{proof}
  From Lemma \ref{lemma41}-\ref{lemma43}, we know that there exists a unique $\bar\varrho>0$ such that the blow-up time of the Riccati equation is $t^K_{\bar\varrho}=0$. Then we would like to prove that $\bar\rho=1-\bar\varrho$ is an eigenvalue. In fact, if we can find a non-trivial solution of the Hamiltonian system \eqref{3825} corresponding to $\bar\rho$, then the desired result is obvious. According to Lemma \ref{lemma43}, the following set:
  \begin{equation*}
    N=\Big\{\tilde x\ |\ \tilde x\in \mathbb R^n, \|\tilde x\|=1,\Big\langle\lim_{t\searrow 0}K^{-1}_t(\bar\varrho) \tilde x,\tilde x\Big\rangle =0 \Big\},
  \end{equation*}
  is nonempty. Define $I_1=[0,\frac{T}{2}]$ and $I_2=[\frac{T}{2},T]$. Noticing that $\tilde K_t(\varrho)$ is well-defined on $[0,T)$, i.e., there is no blow-up time of $\tilde K_t(\varrho)$ on $[0,T)$. Then we introduce the stochastic Hamiltonian system with coefficient matrix \eqref{3831} as follows,
  \begin{equation}\label{3832}
  \left\{
  \begin{aligned}
    &\mathrm{d}\tilde x_t=[\tilde H_{21}(t)\tilde x_t+\tilde H_{22}(t)\tilde y_t+\tilde H_{23}(t) \tilde z_t+\tilde H_{24}(t)\tilde \theta_{t}]\mathrm{d}t\\
    &\qquad\quad+[\tilde H_{31}(t)\tilde x_t+\tilde H_{32}(t)\tilde y_t+\tilde H_{33}(t) \tilde z_t]\mathrm{d}B_t\\
    &\qquad\quad +[\tilde H_{41}(t)\tilde x_{t-}+\tilde H_{42}(t)\tilde y_{t-} +\tilde H_{44}(t)\tilde \theta_{t}]\mathrm{d}\tilde V_t,\\
    &\mathrm{d}\tilde y_t=-[\tilde H_{11}(t)\tilde x_t+\tilde H_{12}(t)\tilde y_t+\tilde H_{13}(t) \tilde z_t+\tilde H_{14}(t)\tilde \theta_{t}]\mathrm{d}t+\tilde z_t\mathrm{d}B_t+\tilde\theta_{t}\mathrm{d}\tilde V_t,\\
    &\tilde x_0=\tilde x\in N,\qquad \tilde y_{\frac{T}{2}}=\tilde K_{\frac{T}{2}}\tilde x_{\frac{T}{2}},\qquad t\in I_1.
  \end{aligned}
  \right.
\end{equation}
Since $K_t(\varrho)$ and $\tilde K_{t}(\varrho)$ is non-negative, $-H_{33}^{-1}(t)+K_t$, $-H_{44}^{-1}(t)+K_t$, $-H_{33}^{-1}(t)+\tilde K_t$ and $-H_{44}^{-1}(t)+\tilde K_t$ are all invertible. Then $I_n-K_tH_{33}(t)$ and $I_n-K_tH_{44}(t)$ are invertible. And the same conclusion can be derived about $I_n-\tilde K_t\tilde H_{33}(t)$ and $I_n-\tilde K_t\tilde H_{44}(t)$. Hence, we can obtain the solution of Hamiltonian system \eqref{3832} by Lemma \ref{lemma3.2}. More specifically, FBSDE \eqref{3832} can be solved by
\begin{equation*}
  \left\{
  \begin{aligned}
    &\mathrm{d}\tilde x_t=[\tilde H_{21}(t)+\tilde H_{22}(t)\tilde K_t+\tilde H_{23}(t) \tilde L_t+\tilde H_{24}(t)\tilde P_{t}]\tilde x_t\mathrm{d}t\\
    &\qquad\quad+[\tilde H_{31}(t)+\tilde H_{32}(t)\tilde K_t+\tilde H_{33}(t) \tilde L_t]\tilde x_t\mathrm{d}B_t\\
    &\qquad\quad +[\tilde H_{41}(t)+\tilde H_{42}(t)\tilde K_{t} +\tilde H_{44}(t)\tilde P_{t}]\tilde x_{t-}\mathrm{d}\tilde V_t,\\
    &\tilde x_0=\tilde x\in N,\\
    &\tilde y_t=\tilde K_t\tilde x_t,\quad \tilde z_t=\tilde L_t\tilde x_t,\quad \tilde \theta_t=\tilde P_t\tilde x_t,\quad t\in I_1.
  \end{aligned}
  \right.
\end{equation*}
Thus one can get $\tilde y_{\frac{T}{2}}$. Let $x_{\frac{T}{2}}=\tilde y_{\frac{T}{2}}$, we can solve the following Hamiltonian system on the interval $I_2$ as follows,
\begin{equation*}
  \left\{
  \begin{aligned}
    &\mathrm{d}x_t=[H_{21}(t)x_t+\varrho H_{22}(t)y_t +\varrho H_{23}(t) z_t+\varrho H_{24}(t) \theta_{t}]\mathrm{d}t\\
    &\qquad\quad +[\varrho H_{31}(t)x_t+\varrho H_{32}(t)y_t+H_{33}(t)z_t]\mathrm{d}B_t\\
    &\qquad\quad+[\varrho H_{41}(t)x_{t-}+\varrho H_{42}(t)y_{t-}+H_{44}(t)\theta_{t}]\mathrm{d}\tilde V_t,\\
    &\mathrm{d}y_t=-[H_{11}(t)x_t+H_{12}(t)y_t+\varrho H_{13}(t)z_t+\varrho H_{14}(t) \theta_{t}]\mathrm{d}t+z_t\mathrm{d}B_t+\theta_{t}\mathrm{d}\tilde V_t,\\
    &x_{\frac{T}{2}}=\tilde y_{\frac{T}{2}},\qquad y_T=0,\qquad t\in I_2.
  \end{aligned}
  \right.
\end{equation*}
By the similar argument, the above FBSDE can be solved by
\begin{equation*}
  \left\{
  \begin{aligned}
    &\mathrm{d}x_t=[H_{21}(t)+\varrho H_{22}(t)K_t +\varrho H_{23}(t) L_t+\varrho H_{24}(t) P_{t}]x_t\mathrm{d}t\\
    &\qquad\quad +[\varrho H_{31}(t)+\varrho H_{32}(t)K_t+H_{33}(t)L_t]x_t\mathrm{d}B_t\\
    &\qquad\quad+[\varrho H_{41}(t)+\varrho H_{42}(t)K_t+H_{44}(t)P_t]x_{t-}\mathrm{d}\tilde V_t,\\
    &x_{\frac{T}{2}}=\tilde y_{\frac{T}{2}},\\
    & y_t=K_tx_t,\quad z_t=L_tx_t,\quad \theta_t=P_tx_t,\quad  t\in I_2.
  \end{aligned}
  \right.
\end{equation*}
By virtue of the dual Legendre transformation, we know that the quadruple $(x_t,y_t,z_t,\theta_t)$ defined by
\begin{equation*}
  \left\{
  \begin{aligned}
    &(\tilde y_t,\tilde x_t, \tilde H_{31}(t)\tilde x_t+\tilde H_{32}(t)\tilde y_t+\tilde H_{33}(t)\tilde z_t,\tilde H_{41}(t)\tilde x_t+\tilde H_{42}(t)\tilde y_t+\tilde H_{44}(t)\tilde \theta_t),\qquad t\in I_1,\\
    &(x_t,y_t,z_t,\theta_t),\qquad t\in I_2,
  \end{aligned}
  \right.
\end{equation*}
is a non-trivial solution of the Hamiltonian system \eqref{3825}. By virtue of Lemma \ref{lemma3.2}, all non-trivial solutions satisfy $\tilde y_t=\tilde K_t(\bar\varrho)\tilde x_t$, $\in I_1$ and $y_t=K_t(\bar\varrho) x_t$, $t\in I_2$. Moreover, for any given $\bar\rho$, $x_0=\tilde y_0=\tilde K_0(\bar\varrho)\tilde x=0$. The space of eigenfunctions consists of all linear combinations of these solutions. Thus the dimensional of the eigenfunction space is less than $\dim \ker \tilde K_0(\bar\varrho)\leq n$.

Finally, if $\tilde\rho<\bar\rho$, then the solution $K_t(\tilde\rho)\geq 0$ of Riccati equation \eqref{3725} is uniformly bounded on $[0,T]$ by Lemma \ref{lemma41}. Thus $\det[I_n-K_tH_{33}(t)]^{-1}\ne 0$ and $\det[I_n-K_tH_{44}(t)]^{-1}\ne 0$ on $[0,T]$. According to Lemma \ref{lemma3.2}, Hamiltonian system \eqref{3825} admits a unique trivial solution, which means that $\bar\rho$ is the first eigenvalue.
\end{proof}

\section{One-dimensional Case}
In this section, we study the eigenvalue problem in one-dimensional case and present some more concrete conclusions. The perturbation matrix $\bar H$ has the following form,
\begin{equation*}
  \bar H(t)=\left(\begin{matrix}
    0&0&0&0\\
    0&\bar H_{22}(t)&0&0\\
    0&0&0&0\\
    0&0&0&0
  \end{matrix}\right).
\end{equation*}
Then Hamiltonian system \eqref{LFBSDE} becomes
\begin{equation}\label{31034}
  \left\{
  \begin{aligned}
    &\mathrm{d}x_t=[H_{21}(t)x_t+ H_{22}(t)y_t-\rho \bar H_{22}(t)y_t + H_{23}(t) z_t+ H_{24}(t) \theta_{t}]\mathrm{d}t\\
    &\qquad\quad +[ H_{31}(t)x_t+ H_{32}(t)y_t+H_{33}(t)z_t]\mathrm{d}B_t\\
    &\qquad\quad+[ H_{41}(t)x_{t-}+ H_{42}(t)y_{t-}+H_{44}(t)\theta_{t}]\mathrm{d}\tilde V_t,\\
    &\mathrm{d}y_t=-[H_{11}(t)x_t+H_{12}(t)y_t+  H_{13}(t)z_t+ H_{14}(t) \theta_{t}]\mathrm{d}t+z_t\mathrm{d}B_t+\theta_{t}\mathrm{d}\tilde V_t,\\
    &x_0=0,\qquad y_T=0.
  \end{aligned}
  \right.
\end{equation}
And we introduce some assumption as follows.

\noindent\textbf{Assumption (H4)}

i) $H^{\top}=H$, $H_{kl}, \bar H_{22}\in C(0,T;\mathbb R)$, $k,l=1,2,3,4$.

ii) $H(t)$ satisfies the condition \eqref{mono} uniformly in $t\in [0,T]$.

iii) $H_{23}(t)=-H_{33}(t)H_{13}(t)$, $H_{24}(t)=-H_{44}(t)H_{14}(t)$, $\bar H_{22}(t)<0$, $t\in [0,T]$.

\begin{remark}
  In order to study the eigenvalue problem of stochastic Hamiltonian system, the decoupling method of this regime-switching FBSDE \eqref{LFBSDE} plays a key role. In this process, it is crucial to study the well-posedness of generalized Riccati equations system \eqref{Riccati1}, which is guaranteed by iii) in Assumption (H4).
\end{remark}
Then the related Riccati equation \eqref{3616} becomes
\begin{equation}\label{3935}
  \left\{
  \begin{aligned}
    &-\dot{k}_t=[2H_{21}(t)+H_{13}^2(t)+H_{14}^2(t)]k_t+H_{11}(t)\\
    &\qquad\qquad +[H_{22}(t)-\rho \bar H_{22}(t)-H_{33}(t)H_{13}^2(t) -H_{44}(t)H_{14}^2(t)]k_t^2,\\
    &k_T=0.
  \end{aligned}
  \right.
\end{equation}
And the dual Hamiltonian system \eqref{3913} becomes
\begin{equation}\label{31136}
  \left\{
  \begin{aligned}
    &\mathrm{d}\tilde x_t=[\tilde H_{21}(t)\tilde x_t+\tilde H_{22}(t)\tilde y_t+\tilde H_{23}(t) \tilde z_t+\tilde H_{24}(t)\tilde \theta_{t}]\mathrm{d}t\\
    &\qquad\quad+[\tilde H_{31}(t)\tilde x_t+\tilde H_{32}(t)\tilde y_t+\tilde H_{33}(t) \tilde z_t]\mathrm{d}B_t\\
    &\qquad\quad +[\tilde H_{41}(t)\tilde x_{t-}+\tilde H_{42}(t)\tilde y_{t-} +\tilde H_{44}(t)\tilde \theta_{t}]\mathrm{d}\tilde V_t,\\
    &\mathrm{d}\tilde y_t=-[\tilde H_{11}(t)\tilde x_t+\tilde H_{12}(t)\tilde y_t+\tilde H_{13}(t) \tilde z_t+\tilde H_{14}(t)\tilde \theta_{t}]\mathrm{d}t+\tilde z_t\mathrm{d}B_t+\tilde\theta_{t}\mathrm{d}\tilde V_t,\\
    &\tilde x_0=0,\qquad \tilde y_T=0,
  \end{aligned}
  \right.
\end{equation}
where
\begin{equation}
  \begin{aligned}
    \tilde H=\left(\begin{matrix}
      \tilde H_{11}&\tilde H_{12}& - H_{23}H_{33}^{-1}&-H_{24}H_{44}^{-1}\\
      \tilde H_{21}&\tilde H_{22}&- H_{13}H_{33}^{-1}&- H_{14}H_{44}^{-1}\\
      -  H_{33}^{-1}H_{32}&- H_{33}^{-1}H_{31}&H_{33}^{-1}&0\\
      - H_{44}^{-1}H_{42}&- H_{44}^{-1}H_{41}&0&H_{44}^{-1}
    \end{matrix}\right).
  \end{aligned}
\end{equation}
with $\tilde H_{11}=H_{23}H_{33}^{-1}H_{32}+H_{24}H_{44}^{-1}H_{42}- H_{22}+\rho \bar H_{22}$, $\tilde H_{12}= H_{23}H_{33}^{-1}H_{31} + H_{24}H_{44}^{-1}H_{41}-H_{21}$, $\tilde H_{21}= H_{13}H_{33}^{-1}H_{32}+ H_{14}H_{44}^{-1}H_{42}-H_{12}$, $\tilde H_{22} = H_{13}H_{33}^{-1}H_{31}+ H_{14}H_{44}^{-1}H_{41}-H_{11}$. Under Assumption (H4), the dual Riccati equation \eqref{22723} can be rewritten as
\begin{equation}\label{31036}
  \begin{aligned}
    &-\dot{\tilde k}_t=-[2H_{21}(t)+H_{13}^2(t)+H_{14}^2(t)]\tilde k_t-H_{11}(t)\tilde k_t^2\\
    &\qquad\qquad -[H_{22}(t)-\rho \bar H_{22}(t)-H_{33}(t)H_{13}^2(t) -H_{44}(t)H_{14}^2(t)].
  \end{aligned}
\end{equation}
Since $\bar H_{22}(t)$ is continuous with respect to (w.r.t) $t$, we can find two negative constant $h_1$ and $h_2$ such that $h_1\leq \bar H_{22}(t)<h_{2}<0$, $t\in [0,T]$.
\subsection{The existence of eigenvalue}
In order to investigate this problem, we first analysis the properties of blow-up time of Riccati equation, which plays a key role of this problem. By the similar method of Lemma \ref{lemma41}, we obtain
\begin{lemma}\label{lemma52}
  Under Assumption (H4), the blow-up time $t^k_{\rho}$ of Riccati equation \eqref{3935} is increasing w.r.t $\rho$. Moreover,
  \begin{equation*}
    \lim_{\rho\nearrow +\infty}t_{\rho}^{k}=T.
  \end{equation*}
\end{lemma}
By Lemma \ref{lemma52}, define $\rho_a:=\{\rho\ |\ \rho\geq 0, t_{\rho}^k>-\infty\}$ and
\begin{equation*}
  \rho_b:=\frac{\min_{t\in [0,T]}[H_{22}(t)-H_{33}(t)H_{13}^2(t)-H_{44}(t)H_{14}^2(t)]}{\max_{t\in [0,T]}\bar H_{22}(t)},
\end{equation*}
then we have $H_{22}(t)-H_{33}(t)H_{13}^2(t)-H_{44}(t)H_{14}^2(t)-\rho \bar H_{22}(t)>0$, $t\in [0,T]$, $\rho\geq \rho_1$. Furthermore, we can obtain the following result by the similar method of Lemma \ref{lemma41}-\ref{lemma43}.
\begin{lemma}\label{lemma53}
  Let Assumption (H4) hold, the blow-up time $t^k_{\rho}$ is continuous and strictly increasing on $(\rho_0,+\infty)$, where $\rho_0=\rho_a\vee \rho_b$.
\end{lemma}
\begin{lemma}\label{lemma54}
  Under Assumption (H4), the blow-up time $t_{\rho}^{\tilde k}$ of dual Riccati equation \eqref{31036} is increasing in $\rho$, and
  \begin{equation*}
    \lim_{\rho\nearrow +\infty}t_{\rho}^{\tilde k}=T.
  \end{equation*}
\end{lemma}
Then let $\tilde\rho_{a}:=\inf\{\rho\ |\ \rho>0, t_{\rho}^{\tilde k}>-\infty\}$, we have
\begin{lemma}\label{lemma55}
  Let Assumption (H4) hold, the blow-up time $t^{\tilde k}_{\rho}$ is continuous and strictly increasing on $(\tilde\rho_0,+\infty)$, where $\tilde\rho_0=\tilde\rho_a\vee \rho_b$.
\end{lemma}
In order to obtain the eigenvalue of Hamiltonian system \eqref{31034}, we first characterize some properties of blow-up time of the terminal time, whose proof is similar to \cite[Lemma 4.5-4.6]{JW2023}.
\begin{lemma}
  For $0\leq \tilde T_1\leq \tilde T_2\leq T$, $k_t^i$ denote the solution of the following Riccati equation, respectively,
  \begin{equation*}
    \left\{
    \begin{aligned}
      &-\dot{k}_t^i=[2H_{21}(t)+H_{13}^2(t)+H_{14}^2(t)]k^i_t+H_{11}(t)\\
    &\qquad\qquad +[H_{22}(t)-\rho \bar H_{22}(t)-H_{33}(t)H_{13}^2(t) -H_{44}(t)H_{14}^2(t)](k^i_t)^2,\\
    &k_{\tilde T_i}^i=0,\qquad t\leq \tilde T_i.
    \end{aligned}
    \right.
  \end{equation*}
  Then $t_{\rho}^{k^1}\leq t_{\rho}^{k^2}$ (if finite). In addition, $t_{\rho}^{k^i}$ is continuous dependent on $\tilde T_i$, $i=1,2$.
\end{lemma}
\begin{lemma}\label{lemma57}
  For $0\leq \tilde T_1\leq \tilde T_2\leq T$, $\tilde k_t^i$ denote the solution of the following Riccati equation, respectively,
  \begin{equation*}
    \left\{
    \begin{aligned}
      &-\dot{\tilde k}^i_t=-[2H_{21}(t)+H_{13}^2(t)+H_{14}^2(t)]\tilde k_t^i-H_{11}(t)(\tilde k_t^i)^2\\
    &\qquad\qquad -[H_{22}(t)-\rho \bar H_{22}(t)-H_{23}(t)H_{13}^2(t) -H_{44}(t)H_{14}^2(t)],\\
    &\tilde k_{\tilde T_i}^i=0,\qquad t\leq \tilde T_i.
    \end{aligned}
    \right.
  \end{equation*}
  Then $t_{\rho}^{\tilde k^1}\leq t_{\rho}^{\tilde k^2}$ (if finite). In addition, $t_{\rho}^{\tilde k^i}$ is continuous dependent on $\tilde T_i$, $i=1,2$.
\end{lemma}
Next we present another main result of this paper.
\begin{theorem}\label{thm58}
  Under Assumption (H4), there exists a sequence eigenvalue $\rho_1<\rho_2<\rho_3<\cdots$ contained in $(\rho_b,+\infty)$, which satisfies $\rho_m\rightarrow +\infty$ as $m\rightarrow +\infty$. Moreover, the dimensional of eigenfunction space corresponding to each $\rho_m$ is $1$.
\end{theorem}
\begin{proof}
  We first defined a sequence of blow-up times $T>t_{\rho}^1>t_{\rho}^2>\cdots$. Let $t_{\rho}^1$  be the blow-up time of Riccati equation \eqref{3935}. By Lemma \ref{lemma52}-\ref{lemma53}, we know that $t_{\rho}^1$ is a strictly increasing and continuous bijective mapping of $\rho$. Then consider the dual Riccati equation  \eqref{31036} with terminal condition $\tilde k_{t_{\rho}^1}=0$ as follows,
  \begin{equation}\label{31137}
  \left\{
  \begin{aligned}
    &-\dot{\tilde k}_t=-[2H_{21}(t)+H_{13}^2(t)+H_{14}^2(t)]\tilde k_t-H_{11}(t)\tilde k_t^2\\
    &\qquad\qquad -[H_{22}(t)-\rho \bar H_{22}(t)-H_{23}(t)H_{13}^2(t) -H_{44}(t)H_{14}^2(t)],\\
    &\tilde k_{t_{\rho}^1}=0.
  \end{aligned}
  \right.
\end{equation}
Denote by $t^2_{\rho}$ the blow-up time of Riccati equation \eqref{31137}, which is also a strictly increasing and continuous bijective mapping of $\rho$ by Lemma \ref{lemma54}-\ref{lemma55}. Let us consider the Riccati equation \eqref{3935} with terminal control $k_{t^2_{\rho}}=0$ and defined the blow-up time by $t^3_{\rho}$, whose properties is similar to $t_{\rho}^1$ and $t_{\rho}^2$. By repeating the above steps, we can define a sequence blow-up time $T>t_{\rho}^0>t_{\rho}^1>t_{\rho}^2>\cdots$ satisfying $\lim_{\rho\nearrow +\infty}t_{\rho}^i=T$, $i=1,2,\cdots$.

For any given $\tilde\rho>\rho_b$, $(\tilde T-t_{\tilde\rho,\tilde T}^k)\wedge (\tilde T-t_{\tilde\rho,\tilde T}^{\tilde k})>0$, where $t_{\tilde\rho,\tilde T}^k$ (resp. $t_{\tilde\rho,\tilde T}^{\tilde k}$) denotes the blow-up time of Riccati equation \eqref{3935} (resp. dual Riccati equation \eqref{31036}) with terminal condition $k_{\tilde T}(\tilde \rho)=0$ (resp. $\tilde k_{\tilde T}(\tilde \rho)=0$). Then there exists $n \in \mathbb N_+\bigcup \{0\}$ and $\rho \geq \lambda_b$, such that
\begin{equation*}
  \begin{aligned}
    T-t_{\rho}^{2n+1}=\sum_{i=1}^{2n+1}(t^{i-1}_{\rho}-t^i_{\rho})=\sum_{i=0}^n(t_{\rho}^{2i}- t^{k}_{\rho,t_{\rho}^{2i}})+\sum_{i=1}^n(t^{2i-1}_{\rho}-t^{\tilde k}_{\rho,t^{2i-1}_{\rho}}) >T.
  \end{aligned}
\end{equation*}
It follows that for this $\rho$ and $n$, $t^{2n+1}_{\rho}<0$. Since $t_{\rho}^{2n+1}$ is strictly increasing continuous bijective mapping and $\lim_{\rho\nearrow+\infty}t_{\rho}^{2n+1}=T$, there exist unique minimizer $\rho_1>\rho_b$ and $2n\in \mathbb N_+\bigcup\{0\}$ such that $t^{2n+1}_{\rho_1}=0$. By Lemma \ref{lemma52}-\ref{lemma57}, we know that $t_{2n+1+2m}$, $m=0,1,2,\cdots$ are also strictly increasing continuous bijective mapping and satisfy $\lim_{\rho\nearrow +\infty}t_{\rho}^{2n+1+2m}=T$. Moreover, there exists a unique $\rho_{m+1}\in (\rho_m,+\infty)$, such that $t_{\rho_m}^{2n+1+2m}=0$. Hence, we derive a sequence $\{\rho_m\}_{m=1}^{\infty}$ satisfying $\rho_b<\rho_1<\rho_2<\cdots$ and $t_{\rho_{m}}^{2n+1+2m}=0$.

Next, we prove that $\{\rho_m\}_{m=1}^{\infty}$ are exactly all the eigenvalues of Hamiltonian system \eqref{31034} on $(\rho_b,+\infty)$. In fact, similar to Theorem \ref{theorem44}, we would like to construct the related eigenfunctions of eigenvalues $\{\rho_m\}_{m=1}^{\infty}$. Noticing that
\begin{equation*}
  0=t^{2n-1+2m}_{\rho_m}<t^{2n-2+2m}_{\rho_m}<\cdots<t_{\rho_m}^2<t_{\rho_m}^1<T,
\end{equation*}
we divide $[0,T]$ into $2m+2n$ parts, $I_1=[t_1,\tilde t_1]:=[0,\frac{t^{2n-2+2m}_{\rho_m}}{2}]$, $I_2=[t_2,\tilde t_2]:=[\frac{t^{2n-2+2m}_{\rho_m}}{2}, \frac{t^{2n-2+2m}_{\rho_m}+t^{2n-3+2m}_{\rho_m}}{2}]$, $\cdots$, $I_{2n-1+2m}=[t_{2n-1+2m},\tilde t_{2n-1+2m}]:=[\frac{t^2_{\rho_m}+t^1_{\rho_m}}{2}, \frac{t^1_{\rho_m}+T}{2}]$, $I_{2n+2m}=[t_{2n+2m},\tilde t_{2n+2m}]:=[\frac{t^1_{\rho_m}+T}{2}, T]$. By the Legendre transformation in Section \ref{3}, we obtain $\tilde k_t(\rho)=k_t(\rho)$ if they are non-zero. According to above processes, one can know that $\tilde k_{\cdot}(\rho)$ exists on $I_1\bigcup I_3\bigcup\cdots \bigcup I_{2m-1+2n}$ and $ k_{\cdot}(\rho)$ exists on $I_2\bigcup I_4\bigcup\cdots \bigcup I_{2m+2n}$. By the similar argument of Theorem \ref{theorem44}, we solve the dual Hamiltonian system \eqref{31136} with boundary condition $\tilde x_0=\tilde x\ne 0$, $\tilde y_{\tilde t_1}=\tilde k_{\tilde t_1}\tilde x_{\tilde t_1}$ by Lemma \ref{lemma3.2} as follows,
\begin{equation*}
  \left\{
  \begin{aligned}
    &\mathrm{d}\tilde x_t=[\tilde H_{21}(t)+\tilde H_{22}(t)\tilde K_t+\tilde H_{23}(t) \tilde L_t+\tilde H_{24}(t)\tilde P_{t}]\tilde x_t\mathrm{d}t\\
    &\qquad\quad+[\tilde H_{31}(t)+\tilde H_{32}(t)\tilde K_t+\tilde H_{33}(t) \tilde L_t]\tilde x_t\mathrm{d}B_t\\
    &\qquad\quad +[\tilde H_{41}(t)+\tilde H_{42}(t)\tilde K_{t} +\tilde H_{44}(t)\tilde P_{t}]\tilde x_{t-}\mathrm{d}\tilde V_t,\\
    &\tilde x_0=\tilde x,\quad  \tilde y_{\tilde t_1}=\tilde K_{\tilde t_1}\tilde x_{\tilde t_1},\quad \tilde z_{\tilde t_1}=\tilde L_{\tilde t_1}\tilde x_{\tilde t_1},\quad \tilde \theta_{\tilde t_1}=\tilde P_{\tilde t_1}\tilde x_{\tilde t_1},\quad t\in I_1.
  \end{aligned}
  \right.
\end{equation*}
Noticing Assumption (H4), we know $\tilde H_{23}=-H_{13}\tilde H_{33}$, $\tilde H_{24}=-H_{14}\tilde H_{44}$. Let $c=\frac{1}{\max_{t\in [0,T]}|H_{13}(t)|^2+1}$, we easily verify that Condition \eqref{weaker} hold, which implies that the dual Hamiltonian system \eqref{31136} admits a unique solution on $I_1$. By the similar method, we can solve the Hamiltonian system \eqref{31034} with boundary condition $x_{t_2}=\tilde y_{\tilde t_1}$, $y_{\tilde t_2}= k_{\tilde t_2}x_{\tilde t_2}$ on $I_2$. We also obtain the unique solution $(x_t,y_t,z_t,\theta_t) =(x_t,k_t,L_tx_t,P_tx_t)$ on $I_2$ by Lemma \ref{lemma3.2}. Then dual Hamiltonian system \eqref{31136} with boundary condition $\tilde x_{t_3}=y_{\tilde t_2}$, $\tilde y_{t_4}=\tilde k_{t_4}\tilde x_{t_4}$ also admits a unique solution on $I_3$ by Lemma \ref{lemma3.2}. By repeating the above steps, we get a unique solution $(x_{\cdot},y_{\cdot},z_{\cdot},\theta_{\cdot})$ of Hamiltonian system \eqref{31034} on $I_2\bigcup I_4\bigcup\cdots \bigcup I_{2m+2n}$  and a unique solution $(\tilde x_{\cdot},\tilde y_{\cdot},\tilde z_{\cdot},\tilde \theta_{\cdot})$ of dual Hamiltonian system \eqref{31136} on $I_1\bigcup I_3\bigcup\cdots \bigcup I_{2m-1+2n}$. By virtue of dual degendre transformation, let
\begin{equation*}
  \left\{
  \begin{aligned}
    &(\tilde y_t,\tilde x_t, \tilde H_{31}(t)\tilde x_t+\tilde H_{32}(t)\tilde y_t+\tilde H_{33}(t)\tilde z_t,\tilde H_{41}(t)\tilde x_t+\tilde H_{42}(t)\tilde y_t+\tilde H_{44}(t)\tilde \theta_t),\quad t\in I_1, I_3,\cdots, I_{2m-1+2n},\\
    &(x_t,y_t,z_t,\theta_t),\quad t\in I_2,I_4,\cdots, I_{2n+2m},
  \end{aligned}
  \right.
\end{equation*}
is a nontrivial solution to Hamiltonian system \eqref{31034} corresponding to eigenvalue $\rho_m$.  From the uniqueness result in Lemma \ref{lemma3.2}, we can prove that the eigenfunction space is just spanned by this eigenfunction. Thus it is of one-dimensional. In addition, we also can prove that no other eigenvalue.
\end{proof}

In Theorem \ref{thm58}, we characterize the eigenvalue on $(\rho_b,+\infty)$, but we did not study that on $(0,\rho_b]$. Next, we introduce some additional assumption to investigate the eigenvalue problem on $\mathbb R_+$.

\noindent \textbf{Assumption (H5)} The coefficients of Hamiltonian system \eqref{31034} satisfy
\begin{equation*}
  4\|H_{11}\|_{\infty}\|H_{22}-\rho_b\bar H_{22}-H_{33}H_{13}^2-H_{44}H_{14}^2\|_{\infty}\leq \|2H_{21}+H_{13}^2+H_{14}^2\|_{\infty}^2<\frac{4}{T^2}.
\end{equation*}

Now we introduce an auxiliary equation as follows,
\begin{equation*}
  \left\{
  \begin{aligned}
    &-\dot{k}_{1,t}=\|2H_{21}+H_{13}^2+H_{14}^2\|_{\infty} k_{1,t}+\|H_{11}\|_{\infty} +\|H_{22}-\rho_b\bar H_{22}-H_{33}H_{13}^2-H_{44}H_{14}^2\|_{\infty} k_{1,t}^2,\\
    &k_{1,T}=0,\qquad \rho \geq \rho_b, \qquad t\leq T,
  \end{aligned}
  \right.
\end{equation*}
According to comparison theorem, we have $0\leq k_t(\rho_b)\leq k_{1,t}(\rho_b)\leq k_{1,t}(\rho)$, $\rho\geq \rho_b$. Then we obtain the following result of the blow-up time of above Riccati equation by the similar method of Lemma \ref{lemma41}.
\begin{lemma}\label{lemma59}
  Under Assumption (H4)-(H5), then $\lim_{\rho\searrow \rho_b}t_{\rho}^{k_1}<0$. Moreover, $\lim_{\rho\searrow \rho_b}t_{\rho}^{k}<0$.
\end{lemma}
By Lemma \ref{lemma59}, $k_{\cdot}(\rho_b)$ exists on $[0,T]$. Since $H_{22}-\rho\bar H_{22}-H_{33}H_{13}^2-H_{44}H_{14}^2\leq H_{22}-\rho_b\bar H_{22}-H_{33}H_{13}^2-H_{44}H_{14}^2$, we obtain $0\leq k_t(\rho)\leq k_t(\rho_b)$ by comparison theorem. Then for any $\rho\in (0,\rho_b]$, $k_{\cdot}(\rho)$ exists on whole $[0,T]$, which means that Hamiltonian system \eqref{31034} admits a unique trivial solution and there is not any other eigenvalue on $(0,\rho_b)$. In summary,
\begin{theorem}
  Under Assumption (H4)-(H5), there exists a sequence eigenvalue $\{\rho_m\}_{m=1}^{\infty}\in \mathbb R_+$, which satisfies $\rho_m\rightarrow +\infty$ as $m\rightarrow +\infty$. Moreover, the dimensional of eigenfunction space corresponding to each $\rho_m$ is $1$.
\end{theorem}

\subsection{The order of growth for the eigenvalue}
In this subsection, we study the order of growth for the eigenvalue $\{\rho_m\}_{m=1}^{\infty}$ in Theorem \ref{thm58}.
\begin{theorem}\label{thm511}
  Suppose Assumption (H4) holds. Let $\{\rho_m\}_{m=1}^{\infty}$ be all eigenvalue contained on $(\rho_b,+\infty)$, then
  \begin{equation*}
    \rho_{m}=O(m^2),\qquad \text{as}\ \  m\rightarrow +\infty.
  \end{equation*}
\end{theorem}
\begin{proof}
  We choose some constant such that for $\forall t\in [0,T]$,
  \begin{equation*}
    \begin{aligned}
      &0<\check H_{11}<H_{11}(t)<\hat H_{11}, \quad &&\check H_{22}<H_{22}(t)<\hat H_{22}<0,\\
      &0\leq \check H_{13} \leq |H_{13}(t)|\leq \hat H_{13},\quad &&\check H_{33}<H_{33}(t)<\hat H_{33}<0,\\
      &0\leq \check H_{14} \leq |H_{14}(t)|\leq \hat H_{14},\quad &&\check H_{44}<H_{44}(t)<\hat H_{44}<0,\\
      &\check H_{21}\leq H_{21}(t)\leq \hat H_{21},\quad &&  \check {\bar H}_{22}<\bar H_{22}(t)<\hat {\bar H}_{22}<0.
    \end{aligned}
  \end{equation*}
  And we assume that $\hat H_{23}:=-\check H_{33}\hat H_{13}$, $\check H_{23}:=-\hat H_{33}\check H_{13}$, $\hat H_{24}:=-\check H_{44}\hat H_{14}$, $\check H_{24}:=-\hat H_{44}\check H_{14}$.

  We first prove that there are $\{\rho_m\}_{m=1}^{\infty}$, such that $\rho_m\leq \check \rho_m$ and $\check \rho_m\sim m^2$ as $m\rightarrow +\infty$. We introduce the stochastic Hamiltonian system with following time-invariance coefficient matrix
  \begin{equation*}
    \left(\begin{matrix}
      \check H_{11}&\check H_{12}&\check H_{13}&\check H_{14}\\
      \check H_{21}& \check H_{22}&\check H_{23}&\check H_{24}\\
      \check H_{31}&\check H_{32}&\hat H_{33}&0\\
      \check H_{41}&\check H_{42}&0&\hat H_{44}
    \end{matrix}\right),
  \end{equation*}
  which satisfies Assumption (H1). Then we consider the eigenvalue problem of following time-invariance Hamiltonian system
  \begin{equation}\label{31340}
  \left\{
  \begin{aligned}
    &\mathrm{d}x_t=[\check H_{21}(t)x_t+\check H_{22}(t)y_t-\rho \hat{\bar H}_{22}(t)y_t +\check H_{23}(t) z_t+ \check H_{24}(t) \theta_{t}]\mathrm{d}t\\
    &\qquad\quad +[ \check H_{31}(t)x_t+ \check H_{32}(t)y_t+\hat H_{33}(t)z_t]\mathrm{d}B_t\\
    &\qquad\quad+[ \check H_{41}(t)x_{t-}+ \check H_{42}(t)y_{t-}+\hat H_{44}(t)\theta_{t}]\mathrm{d}\tilde V_t,\\
    &\mathrm{d}y_t=-[\check H_{11}(t)x_t+\check H_{12}(t)y_t+\check H_{13}(t)z_t+\check H_{14}(t) \theta_{t}]\mathrm{d}t+z_t\mathrm{d}B_t+\theta_{t}\mathrm{d}\tilde V_t,\\
    &x_0=0,\qquad y_T=0.
  \end{aligned}
  \right.
\end{equation}
Let $\check \rho_m$ be the eigenvalue of above Hamiltonian system. Correspondingly, we introduce the following Riccati equation
\begin{equation*}
  \begin{aligned}
    &-\dot{\check k}_t=[2\check H_{21}(t)+\check H_{13}^2(t)+\check H_{14}^2(t)]k_t+\check H_{11}(t)\\
    &\qquad\qquad +[\check H_{22}(t)-\rho \hat{\bar H}_{22}(t)-\hat H_{33}(t)\check H_{13}^2(t) -\hat H_{44}(t)\check H_{14}^2(t)]k_t^2,
  \end{aligned}
\end{equation*}
and dual Riccati equation
\begin{equation*}
  \begin{aligned}
    &-\dot{\check{\tilde k}}_t=-[2\check H_{21}(t)+\check H_{13}^2(t)+\check H_{14}^2(t)]\check {\tilde k}_t-\check H_{11}(t)\check {\tilde k}_t^2\\
    &\qquad\qquad -[\check H_{22}(t)-\rho \check {\bar H}_{22}(t)-\hat H_{33}(t)\check H_{13}^2(t) -\hat H_{44}(t)\check H_{14}^2(t)].
  \end{aligned}
\end{equation*}
Recalling Riccati equation \eqref{3935}, we have $k_t(\rho)\geq \check k_t(\rho)\geq 0$ and $\tilde k_{t}(\rho)\leq \check{\tilde k}_t(\rho)\leq 0$ by comparison theorem. Furthermore, the relationship of the blow-up time of these Riccati equations can also be obtained, $t_{\rho}^{k}\geq t_{\rho}^{\check k}$ and $t^{\tilde k}_{\rho}\geq t^{\check{\tilde k}}_{\rho}$. Moreover, by the similar argument of Theorem \ref{thm58}, for sufficiently large $\rho$, $t^{2n+1+2m}_{\rho}\geq \check t^{2n+1+2m}_{\rho}$, $m=0,1,2,\cdots$, where $\{t^{2n-1+2m}_{\cdot}\}$ and $\{\check t^{2n-1+2m}_{\cdot}\}$ are respectively related to Hamiltonian system \eqref{31034} and \eqref{31340}. Then for sufficiently large $m$, we get $\rho_m\leq \check \rho_m$. Moreover, by the similar method of \cite{JW2021}, we have $\frac{\check\rho_m\hat{\bar H}_{22}}{\check H_{22}}=O(m^2)$ as $m\rightarrow +\infty$.

Next we will prove that there are $\{\hat\rho_m\}_{m=1}^{\infty}$, such that $\rho_m\geq \hat\rho_m$ and $\hat\rho_m\sim m^2$ as $m\rightarrow +\infty$. We introduce the stochastic Hamiltonian system with following time-invariance coefficient matrix
  \begin{equation*}
    \left(\begin{matrix}
      \hat H_{11}&\hat H_{12}&\hat H_{13}&\hat H_{14}\\
      \hat H_{21}&\hat H_{22}&\hat H_{23}&\hat H_{24}\\
      \hat H_{31}&\hat H_{32}&\check H_{33}&0\\
      \hat H_{41}&\hat H_{42}&0&\check H_{44}
    \end{matrix}\right),
  \end{equation*}
  which satisfies Assumption (H1). Then by the similar argument of the first step, we can obtain the desired result.
\end{proof}

\bibliographystyle{plain}
\bibliography{eigenvalues}

\end{document}